\documentclass[12pt]{article}
\usepackage{amsthm}
\usepackage{amsfonts}
\usepackage{amsmath}
\usepackage{amssymb}
\usepackage{hyperref}
\usepackage{graphicx}
\usepackage{enumerate}
\usepackage{mathscinet}
\usepackage{fullpage}

\newtheorem{theorem}{Theorem}

\newtheorem{lemma}[theorem]{Lemma}

\newtheorem{claimn}[theorem]{Claim}
\newtheorem{corollary}[theorem]{Corollary}


\if10     % 11 = on,  10 = off
\usepackage[mathlines]{lineno}
\newcommand*\patchAmsMathEnvironmentForLineno[1]{%
  \expandafter\let\csname old#1\expandafter\endcsname\csname #1\endcsname
  \expandafter\let\csname oldend#1\expandafter\endcsname\csname end#1\endcsname
  \renewenvironment{#1}%
     {\linenomath\csname old#1\endcsname}%
     {\csname oldend#1\endcsname\endlinenomath}}% 
\newcommand*\patchBothAmsMathEnvironmentsForLineno[1]{%
  \patchAmsMathEnvironmentForLineno{#1}%
  \patchAmsMathEnvironmentForLineno{#1*}}%
\AtBeginDocument{%
\patchBothAmsMathEnvironmentsForLineno{equation}%
\patchBothAmsMathEnvironmentsForLineno{align}%
\patchBothAmsMathEnvironmentsForLineno{flalign}%
\patchBothAmsMathEnvironmentsForLineno{alignat}%
\patchBothAmsMathEnvironmentsForLineno{gather}%
\patchBothAmsMathEnvironmentsForLineno{multline}%
}
\linenumbers
\fi
\newcommand{\Ss}{{\mathcal S}}

\newcommand{\ntt}{n^t}
\newcommand{\nss}{n^s}
\newcommand{\nt}[1]{n^t_#1} 
\newcommand{\ns}[1]{n^s_#1} 
\newcommand{\Sum}[1]{n^s_#1+n^t_#1} 

\newcommand{\ourURL}{\url{http://orion.math.iastate.edu/lidicky/pub/9cyc/}}

\title{$3$-coloring triangle-free planar graphs with a precolored $9$-cycle}
\author{
Ilkyoo Choi$^{1}$\thanks{Department of Mathematics, Hankuk University of Foreign Studies, Yongin-si, Gyeonggi-do, Republic of Korea, E-mail: {\tt mailto:ilkyoo@hufs.ac.kr}}
\and
Jan Ekstein$^{2}$\thanks{University of West Bohemia, Czech Republic, E-mail: {\tt ekstein@kma.zcu.cz}}
\and
P\v{r}emysl Holub$^{2}$\thanks{University of West Bohemia, Czech Republic, E-mail: {\tt holubpre@kma.zcu.cz}}
\and
Bernard Lidick\'{y}$^{3}$\thanks{Iowa State University, USA, E-mail: {\tt lidicky@iastate.edu}.}
}

\date{\today}

\date{\today}
\begin{document}
\maketitle

\begin{abstract}
Given a triangle-free planar graph $G$ and a $9$-cycle $C$ in $G$, we characterize situations where a $3$-coloring of $C$ does not extend to a proper $3$-coloring of $G$. 
This extends previous results when $C$ is a cycle of length at most $8$.
\end{abstract}

\section{Introduction}
Given a graph $G$, let $V(G)$ and $E(G)$ denote the vertex set and the edge set of $G$, respectively. 
We will also use $|G|$ for the size of $E(G)$. 
A \emph{proper $k$-coloring} of a graph $G$ is a function $\varphi:V(G) \rightarrow \{1,2,\ldots, k\}$
such that $\varphi(u)\neq\varphi(v)$ for each edge $uv\in E(G)$.
A graph $G$ is \emph{$k$-colorable} if there exists a proper $k$-coloring of $G$, and the minimum $k$ where $G$ is $k$-colorable is the {\it chromatic number} of $G$.

Garey and Johnson~\cite{1979GaJo} proved that deciding if a graph is $k$-colorable is NP-complete even when $k=3$. 
Moreover, deciding if a graph is $3$-colorable is still NP-complete when restricted to planar graphs~\cite{1980Da}.
Therefore, even though planar graphs are $4$-colorable by the celebrated Four Color Theorem~\cite{1977ApHa,1977ApHaKo,1997RoSaSeTh}, finding sufficient conditions for a planar graph to be $3$-colorable has been an active area of research. 
A landmark result in this area is Gr{\"{o}}tzsch's Theorem~\cite{grotzsch1959}, which is the following:

\begin{theorem}[\cite{grotzsch1959}]\label{thm-grotzsch}
Every triangle-free planar graph is $3$-colorable. 
\end{theorem}

We direct the readers to a nice survey by Borodin~\cite{2013Bo} for more results and conjectures regarding $3$-colorings of planar graphs. 

A graph $G$ is \emph{$k$-critical} if it is not $(k-1)$-colorable but every proper subgraph of $G$ is $(k-1)$-colorable. 
Critical graphs are important since they are (in a certain sense) the minimal obstacles in reducing the chromatic number of a graph. 
Numerous coloring algorithms are based on detecting critical subgraphs.
Despite its importance, there is no known characterization of $k$-critical graphs when $k\geq 4$. 
On the other hand, there has been some success regarding $4$-critical planar graphs. 
Extending Theorem~\ref{thm-grotzsch}, the Gr\"unbaum--Aksenov Theorem~\cite{1974Ak,1997Bo,1963Gr} states that a planar graph with at most three triangles is $3$-colorable, and we know that there are infinitely many $4$-critical planar graphs with four triangles. 
Borodin, Dvo\v r\'ak, Kostochka, Lidick\'y, and Yancey~\cite{00BoDvKoLiYa} were able to characterize all $4$-critical planar graphs with four triangles.

Given a graph $G$ and a proper subgraph $C$ of $G$, we say $G$ is \emph{$C$-critical for $k$-coloring} if for every proper subgraph $H$ of $G$ where $C\subseteq H$, there exists a proper $k$-coloring of $C$ that extends to a proper $k$-coloring of $H$, but does not extend to a proper $k$-coloring of $G$.
Roughly speaking, a $C$-critical graph for $k$-coloring is a minimal obstacle when trying to extend a proper $k$-coloring of $C$ to a proper $k$-coloring of the entire graph. 
Note that $(k+1)$-critical graphs are exactly the $C$-critical graphs for $k$-coloring with $C$ being the empty graph.

In the proof of Theorem~\ref{thm-grotzsch}, Gr\"otzsch actually proved that any proper coloring of a $4$-cycle or a $5$-cycle extends to a proper $3$-coloring of a triangle-free planar graph. 
This implies that there are no triangle-free planar graphs that are $C$-critical for $3$-coloring when $C$ is a face of length $4$ or $5$. 
This sparked the interest of characterizing triangle-free planar graphs that are $C$-critical for $3$-coloring when $C$ is a face of longer length. 
Since we deal with $3$-coloring triangle-free planar graphs in this paper, from now on, we will write ``$C$-critical'' instead of ``$C$-critical for $3$-coloring'' for the sake of simplicity.

The investigation was first done on planar graphs with girth $5$.
Walls~\cite{1999Wa} and Thomassen~\cite{2003Th} independently characterized $C$-critical planar graphs with girth $5$ when $C$ is a face of length at most $11$. 
The case when $C$ is a $12$-face was initiated in~\cite{2003Th}, but a complete characterization was given by Dvo\v r\'ak and Kawarabayashi in~\cite{00DvKa}.
Moreover, a recursive approach to identify all $C$-critical planar graphs with girth $5$ when $C$ is a face of any given length is given in~\cite{00DvKa}.
Dvo\v r\'ak and Lidick\'y~\cite{14DvLiSurf} implemented the algorithm from~\cite{00DvKa} and used a computer to generate all $C$-critical graphs with girth $5$ when $C$ is a face of length at most $16$.
The generated graphs were used to reveal some structure of $4$-critical graphs on surfaces without short contractible cycles.
It would be computationally feasible to generate graphs with girth $5$ even when $C$ has length greater than 16.

The situation for planar graphs with girth $4$, which are triangle-free planar graphs, is more complicated since the list of $C$-critical graphs is not finite when $C$ has size at least 6.
We already mentioned that there are no $C$-critical triangle-free planar graphs when $C$ is a face of length $4$ or $5$. 
An alternative proof of the case when $C$ is a $5$-face was given by Aksenov~\cite{1974Ak}.
Gimbel and Thomassen~\cite{1997GiTh} not only showed that there exists a $C$-critical triangle-free planar graph when $C$ is a $6$-face, but also characterized all of them. 
A \emph{$k^-$-cycle, $k^+$-cycle} is a cycle of length at most $k$, at least $k$, respectively.  A cycle $C$ in a graph $G$ is \emph{separating} if $G-C$ has more connected components than $G$.
\begin{theorem}[Gimbel and Thomassen~\cite{1997GiTh}]\label{thm-gimbel}
Let $G$ be a connected triangle-free plane graph with outer face bounded by a $6^-$-cycle $C=c_1c_2\cdots$.
% of length at most $6$.
The graph $G$ is $C$-critical if and only if $C$ is a $6$-cycle, all internal faces of $G$ have length exactly four
and $G$ contains no separating $4$-cycles.  Furthermore, if $\varphi$ is a $3$-coloring of $C$
that does not extend to a $3$-coloring of $G$, then $\varphi(c_1)=\varphi(c_4)$, $\varphi(c_2)=\varphi(c_5)$, and $\varphi(c_3)=\varphi(c_6)$.
\end{theorem}
Aksenov, Borodin, and Glebov~\cite{2003AkBoGl} independently proved the case when $C$ is a $6$-face using the discharging method, and also characterized all $C$-critical triangle-free planar graphs when $C$ is a $7$-face in~\cite{2004AkBoGl}.
The case where $C$ is a $7$-face was used in~\cite{00BoDvKoLiYa}.

\begin{figure}
\center{\includegraphics{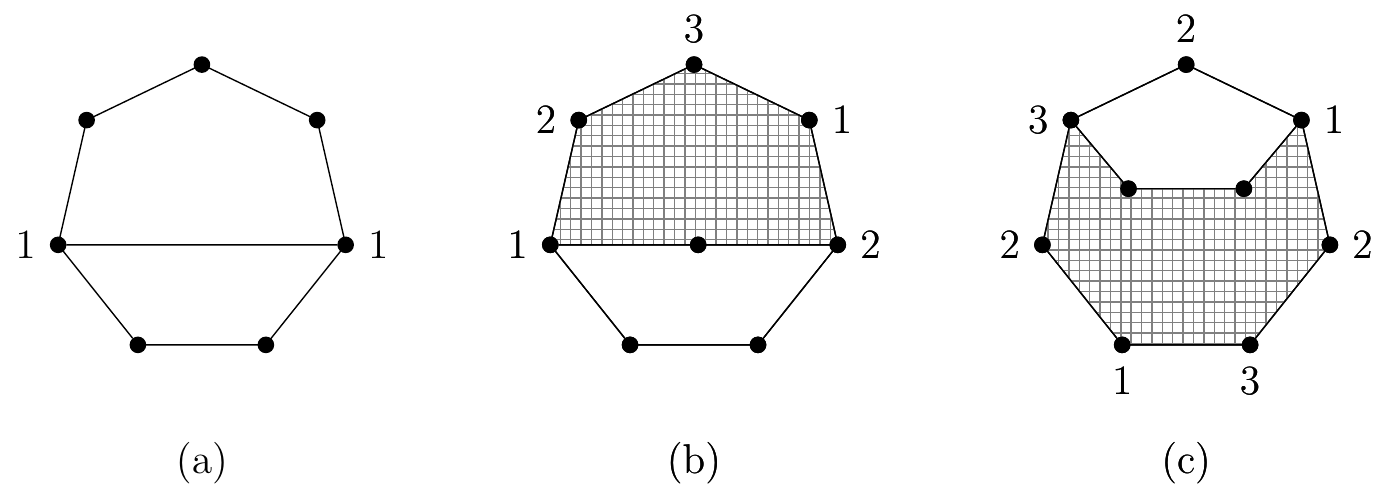}}
\caption{Critical graphs with a precolored $7$-face.}\label{fig-7cyc}
\end{figure}

\begin{theorem}[Aksenov, Borodin, and Glebov~\cite{2004AkBoGl}]\label{thm:7cycleprecise}
Let $G$ be a connected triangle-free plane graph with outer face bounded by a $7$-cycle $C=c_1\cdots c_7$.  The graph $G$ is $C$-critical
and $\psi$ is a $3$-coloring of $C$ that does not extend to a $3$-coloring of $G$ if and only if $G$ contains no separating $5^-$-cycles 
and one of the following propositions is satisfied up to relabelling of vertices (see Figure~\ref{fig-7cyc} for an illustration).
\begin{itemize}
\item[\textrm(a)] The graph $G$ consists of $C$ and the edge $c_1c_5$, and $\psi(c_1)=\psi(c_5)$.
\item[\textrm(b)] The graph $G$ contains a vertex $v$ adjacent to $c_1$ and $c_4$, the cycle $c_1c_2c_3c_4v$ bounds a $5$-face and every face drawn inside
the $6$-cycle $vc_4c_5c_6c_7c_1$ has length four; furthermore, $\psi(c_4)=\psi(c_7)$ and $\psi(c_5)=\psi(c_1)$.
\item[\textrm(c)] The graph $G$ contains a path $c_1uvc_3$ with $u,v\not\in V(C)$, the cycle $c_1c_2c_3vu$ bounds a $5$-face and every face drawn inside
the $8$-cycle $uvc_3c_4c_5c_6c_7c_1$ has length four; furthermore, $\psi(c_3)=\psi(c_6)$, $\psi(c_2)=\psi(c_4)=\psi(c_7)$, and $\psi(c_1)=\psi(c_5)$.
\end{itemize}
\end{theorem}

Dvo\v r\'ak and Lidick\'y~\cite{00DvLi} used a correspondence of nowhere-zero flows and colorings to give simpler proofs of the case when $C$ is either a $6$-face or a $7$-face, and also characterized $C$-critical triangle-free planar graphs when $C$ is an $8$-face.
For a plane graph $G$, let $S(G)$ denote the set of multisets of  lengths of internal faces of $G$ with length at least $5$.

\begin{theorem}[Dvo\v r\'ak and Lidick\'y~\cite{00DvLi}]\label{thm:8cycleprecise}
Let $G$ be a connected triangle-free plane graph with outer face bounded by an $8$-cycle $C$.
The graph $G$ is $C$-critical
if and only if $G$ contains no separating $5^-$-cycles, the interior of every non-facial $6$-cycle contains only $4$-faces, and one of the following propositions is satisfied (see Figure~\ref{fig-8cycle} for an illustration).
\begin{itemize}
 \item[\textrm{(a)}] $S(G)=\emptyset$.
 \item[\textrm{(b)}] $S(G)=\{6\}$ and the $6$-face of $G$ intersects $C$ in a path of length at least one.
 \item[\textrm{(c)}] $S(G)=\{5,5\}$ and each of the $5$-faces of $G$ intersects $C$ in a path of length at least two.
 \item[\textrm{(d)}] $S(G)=\{5,5\}$ and the vertices of $C$ and the $5$-faces $f_1$ and $f_2$ of $G$ can be labelled
 in clockwise order along their boundaries so that $C=c_1c_2\cdots c_8$, $f_1=c_1v_1zv_2v_3$, and $f_2=zw_1c_5w_2w_3$
 (where $w_1$ can be equal to $v_1$, $v_1$ can be equal to $c_2$, etc).
\end{itemize}
\end{theorem}

\begin{figure}
\center{\includegraphics[scale=1]{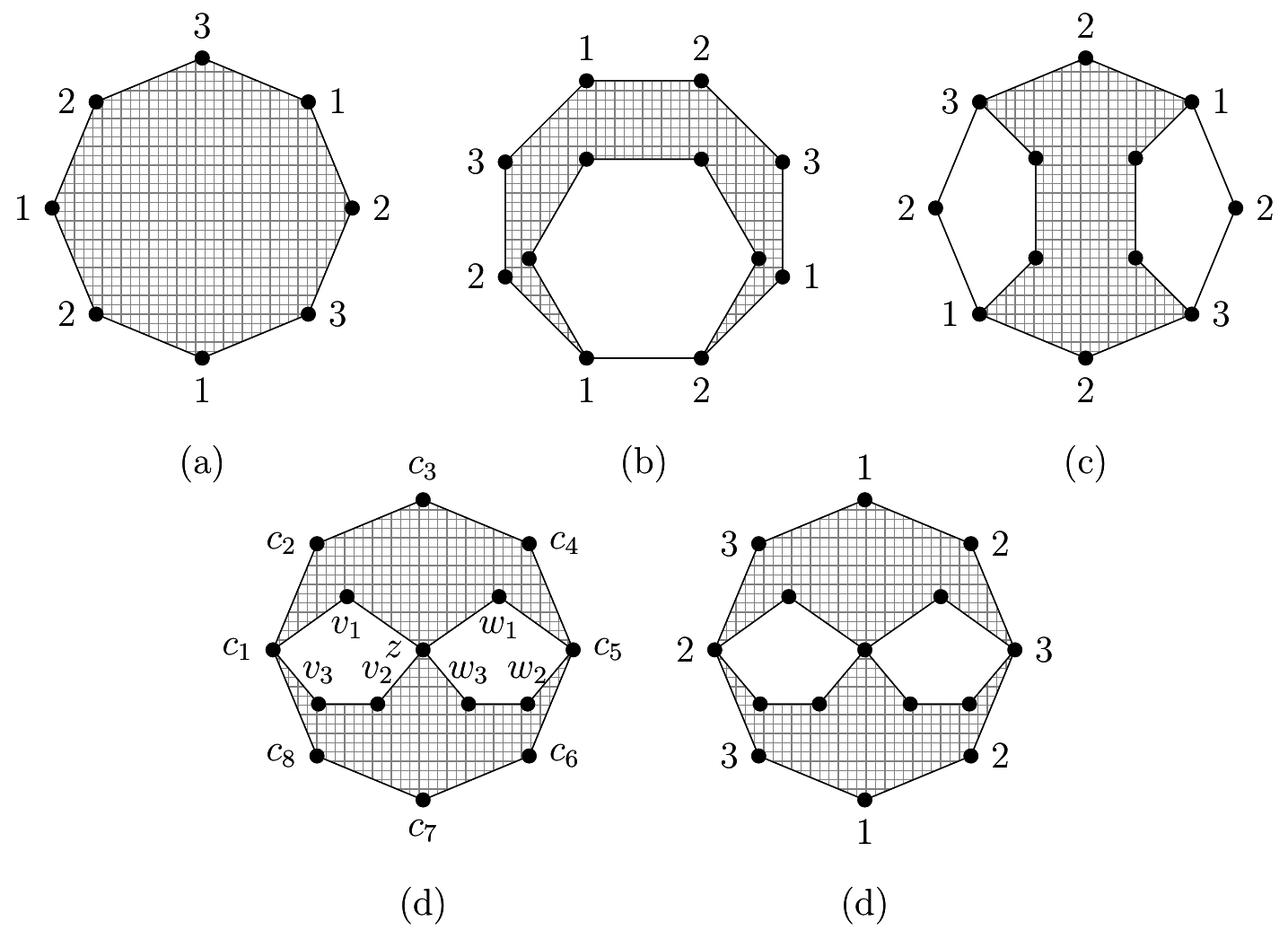}}
\caption{Graphs described by Theorem~\ref{thm:8cycleprecise} and examples of $3$-colorings of $C$ that do not extend.}\label{fig-8cycle}
\end{figure}

Theorem~\ref{thm:8cycleprecise} has the following corollary that was not explicitly stated in~\cite{00DvLi}.

\begin{corollary}[\cite{00DvLi}]\label{corr:8cycleprecise}
Let $G$ be a triangle-free plane graph and let $v$ be a vertex of degree 4 in $G$. Then there exists a proper $3$-coloring of $G$ where all neighbors of $v$ are colored the same.
\end{corollary}
The corollary can be proven by splitting $v$ into four vertices of degree 2 that are in one $8$-face $F$ and precoloring $F$ by two colors, see Figure~\ref{fig-split}.

\begin{figure}
\center{\includegraphics[scale=1]{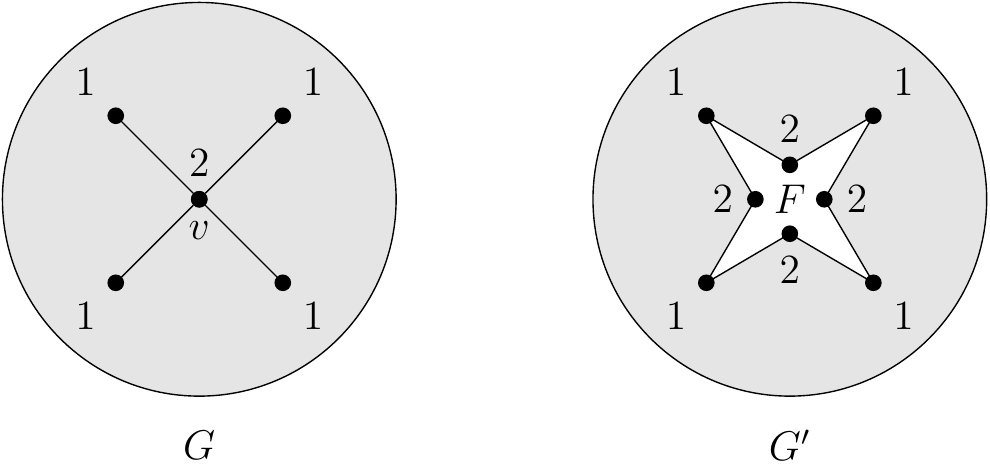}}
\caption{The coloring of a graph $G$, where all neighbors of a 4-vertex  $v$ have the same color, can be obtained by extending a precoloring of an $8$-face $F$ in $G'$, where $G'$ is obtained from $G$ by splitting $v$ into four vertices of degree 2.}\label{fig-split}
\end{figure}

In this paper, we push the project further and characterize all $C$-critical triangle-free planar graphs when $C$ is a $9$-face.

\begin{theorem}\label{thm:9cycleprecise}
Let $G$ be a connected triangle-free plane graph with outer face bounded by a $9$-cycle $C$.  The graph $G$ is $C$-critical for $3$-coloring
if and only if  
for every non-facial $8^-$-cycle of $K$ the subgraph of $G$ drawn in the closed disk bounded by $K$ is $K$-critical
and one of the following propositions is satisfied 
(see Figure~\ref{fig-all-critical-graphs} for an illustration).
\begin{enumerate}[$(a)$]
\item $S(G)=\{5\}$ and the $5$-face of $G$ intersects $C$ in a path of length at least two.
\item $S(G)=\{7\}$.
\item $S(G)=\{5, 6\}$ and the $5$-face and $6$-face of $G$ intersects $C$ in a path of length at least two and one, respectively.
\item $S(G)=\{5, 6\}$ and $G$ is depicted as $(d1)$ or $(d2)$ in Figure~\ref{fig-all-critical-graphs}.
\item $S(G)=\{5, 5, 5\}$ and $G$ is depicted as $(Bij)$ in Figure~\ref{fig-all-critical-graphs} for all $i,j$.
\item $G$ contains a chord.
\end{enumerate}
\end{theorem}

\def\tempL{0.19}
\begin{figure}
\begin{center}
\includegraphics[width=\tempL\textwidth,page=1]{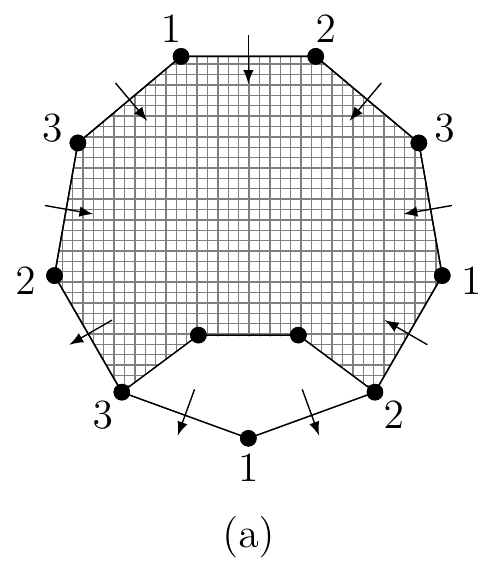}
\includegraphics[width=\tempL\textwidth,page=2]{fig-criticals/fig-7-critical}
\includegraphics[width=\tempL\textwidth,page=1]{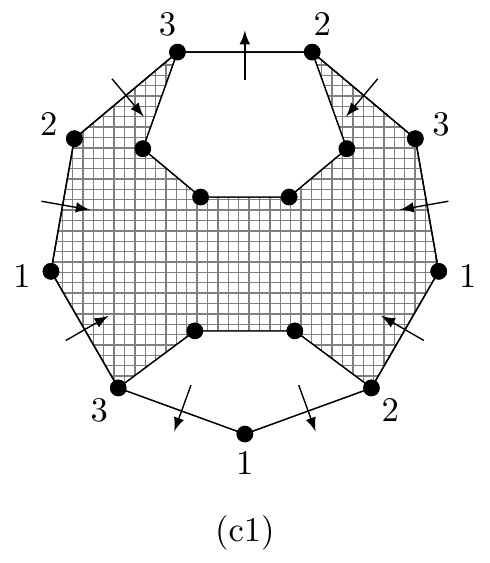}
\includegraphics[width=\tempL\textwidth,page=2]{fig-criticals/fig-65twocutsnocross}
\includegraphics[width=\tempL\textwidth,page=1]{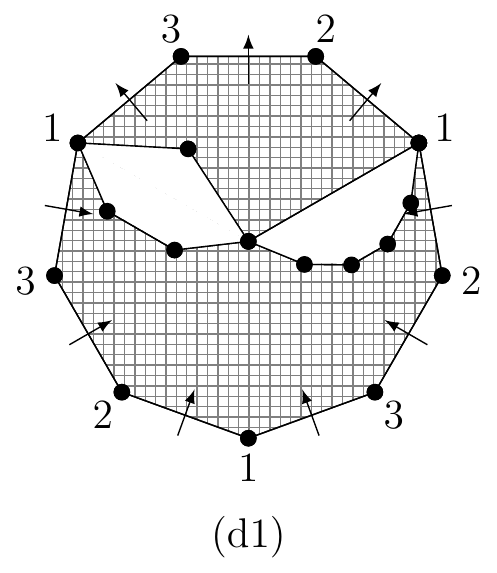}
\includegraphics[width=\tempL\textwidth,page=1]{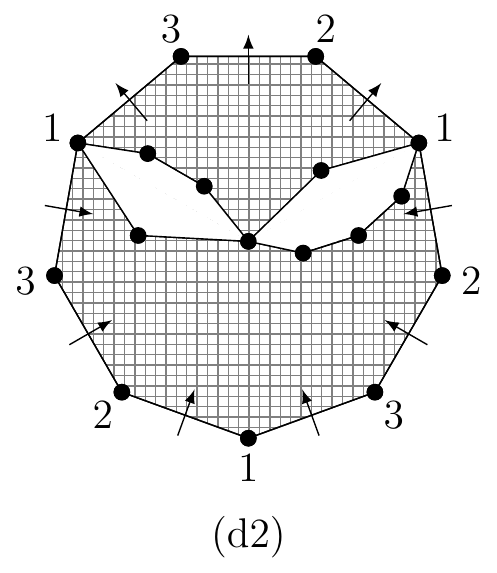}
\includegraphics[width=\tempL\textwidth,page=1]{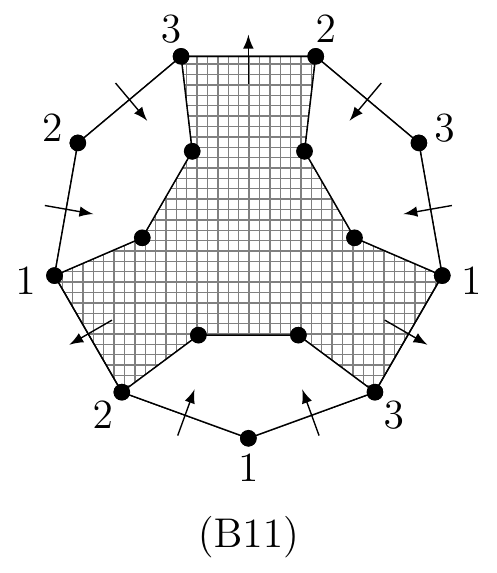}
\includegraphics[width=\tempL\textwidth,page=2]{fig-criticals/fig-555-critical}
\includegraphics[width=\tempL\textwidth,page=3]{fig-criticals/fig-555-critical} 
\includegraphics[width=\tempL\textwidth,page=7]{fig-criticals/fig-555-critical}
\includegraphics[width=\tempL\textwidth,page=8]{fig-criticals/fig-555-critical}
\includegraphics[width=\tempL\textwidth,page=9]{fig-criticals/fig-555-critical}
\includegraphics[width=\tempL\textwidth,page=10]{fig-criticals/fig-555-critical}
\includegraphics[width=\tempL\textwidth,page=11]{fig-criticals/fig-555-critical}
\includegraphics[width=\tempL\textwidth,page=12]{fig-criticals/fig-555-critical}
\includegraphics[width=\tempL\textwidth,page=13]{fig-criticals/fig-555-critical}
\includegraphics[width=\tempL\textwidth,page=15]{fig-criticals/fig-555-critical}
\includegraphics[width=\tempL\textwidth,page=14]{fig-criticals/fig-555-critical}
\includegraphics[width=\tempL\textwidth,page=1]{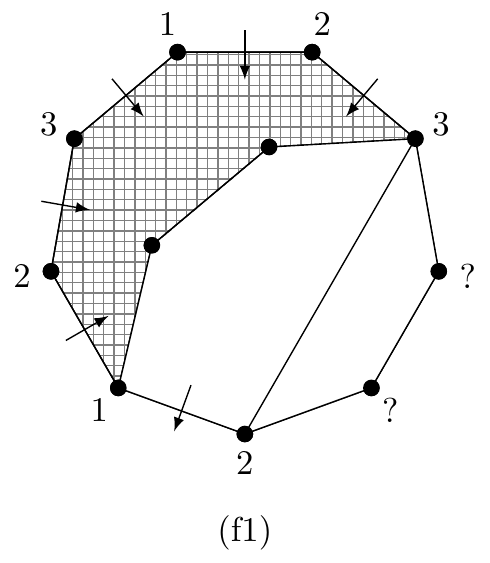}
\includegraphics[width=\tempL\textwidth,page=2]{fig-criticals/fig-chord-critical}

\end{center}
\caption{All $C$-critical triangle-free plane graphs where $C$ is a 9-cycle bounding the outer face.
Note that each figure actually represents infinitely many graphs, including ones that can be obtained by identifying some of the depicted vertices. The arrows correspond to source edges and sink edges that are defined in the Preliminaries.
}
\label{fig-all-critical-graphs}
\end{figure}

The proof of Theorem~\ref{thm:9cycleprecise} involves enumerating all integer solutions 
to several small sets of linear constraints. It would be possible to solve them by hand but we have decided
to use computer programs to enumerate the solutions.
Both computer programs and enumerations of the solutions are available online on arXiv and at
\ourURL.

\section{Preliminaries}
Our proof of Theorem~\ref{thm:9cycleprecise} uses the same method as Dvo\v{r}\'ak and Lidick\'{y}~\cite{00DvLi}.
The main idea is to use the correspondence between colorings of a plane graph $G$ and flows in the dual of $G$.
In this paper, we give only a brief description of the correspondence and state Lemma~\ref{lemma5} from \cite{00DvLi}, which is used throughout this paper.
A more detailed and general description can be found in~\cite{00DvLi}.

Let $G^{\star}$ denote the dual of a 3-colorable plane graph $G$.
Let $\varphi$ be a proper 3-coloring of the vertices of $G$ by colors $\{1,2,3\}$. 
For every edge $uv$ of $G$, we orient the corresponding edge $e$ in $G^{\star}$ in the following way.
Let $e$ have endpoints $f,h$ in $G^\star$, where $f$,$v$,$h$ is in the clockwise order from vertex $u$
in the drawing of $G$. The edge $e$ will be oriented from $f$ to $h$ if $(\varphi(u),\varphi(v)) \in \{(1,2),(2,3),(3,1)\}$, and from $h$ to $f$ otherwise. See Figure~\ref{fig-network} for an example of the orientation.

\begin{figure}[ht]
\begin{center}
\includegraphics[scale=1]{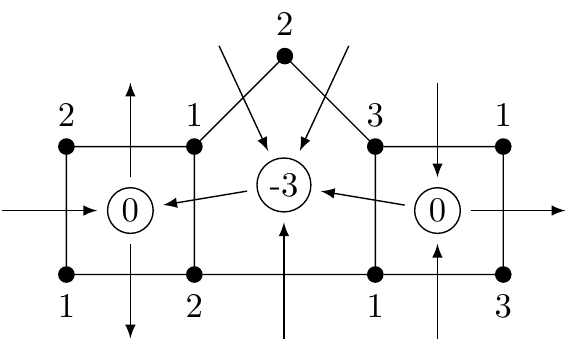}
\end{center}
\caption{A $3$-coloring of a graph $G$ and the corresponding orientation of the edges in $G^\star$.
 }
\label{fig-network}
\end{figure}

Since $\varphi$ is a proper coloring, every edge of $G^\star$ has an orientation. Tutte~\cite{tutteflow} showed that 
this orientation of $G^\star$ defines a nowhere-zero $\mathbb{Z}_3$-flow, which means that the in-degree and the out-degree
of every vertex in $G^\star$ differ by a multiple of three. Conversely, every nowhere-zero $\mathbb{Z}_3$-flow in  $G^\star$
defines a proper 3-coloring of $G$ up to the rotation of colors.

Let $h$ be the vertex in $G^\star$ corresponding to the outer face of $G$. Edges oriented away from $h$ are called \emph{source edges} and the edges oriented towards $h$ are called \emph{sink edges}. The orientations
of edges incident to $h$ depend only on the coloring of $C$, where $C$ is the cycle bounding the outer face of $G$.
Denote by $\nss$ the number of source edges and by $\ntt$ the number of sink edges.
For a subgraph $Z$ of $G$ or a subset $Z$ of $E(G)$, we will use $\ns Z$ and $\nt Z$ to denote the number of source edges and sink edges in $G^\star$ whose dual is in $Z$, respectively. 
Recall that only edges in $C$ have source edges or sink edges in the dual.

For a vertex $f$ of $G^\star$, let $\delta(f)$ denote the difference of the out-degree and in-degree of $f$.
Possible values of $\delta(f)$ depend on the size of the face corresponding to $f$, denoted by $|f|$.
Clearly $|\delta(f)| \leq |f|$ and $\delta(f)$ has the same parity as $|f|$.
Hence if $|f| = 4$, then $\delta(f) = 0$. 
Similarly, if $|f| \in \{5,7\}$, then $\delta(f) \in \{-3,3\}$
and if $|f| = 6$ then $\delta(f) \in \{-6,0,6\}$.

We call a function $q$ assigning an integer to every internal face $f$ of $G$ a \emph{layout} if $q(f) \leq |f|$, $q(f)$ is divisible by 3, and $q(f)$ has the same parity as $|f|$.
Notice that $q(f)$ satisfies the same conditions as $\delta(f)$. 
Therefore it is sufficient to specify the $q$-values for faces of size at least $5$, since $q(f)=0$ if $f$ is a $4$-face.
A layout $q$ is {\it $\psi$-balanced} if $\nss+m=\ntt$, where $m$ is the sum of the $q$-values over all internal faces of $G$.

Our main tool is the following lemma from \cite{00DvLi}.

\begin{lemma}[\cite{00DvLi}] \label{lemma5} % was Lemma 5
Let $G$ be a connected triangle-free plane graph with outer face $C$ bounded by a cycle and let $\psi$ be a $3$-coloring of $C$ that does not extend to a $3$-coloring of $G$. If $q$ is a $\psi$-balanced layout in $G$, then there exists a subgraph $K_0\subseteq G$ such that either
\begin{itemize}
\item[i)] $K_0$ is a path with both ends in $C$ and no internal vertex in $C$, and if $P$ is a path in $C$ joining the end vertices of $K_0$, $n^s$ is the number of source edges of $P$, $n^t$ is the number of the sink edges of $P$, and $m$ is the sum of the $q$-values over all faces of $G$ drawn in the open disk bounded by the cycle $P+K_0$, then $\vert n^s+m-n^t\vert > \vert K_0\vert$. In particular, $\vert P\vert + \vert m \vert >\vert K_0\vert$. \\
Or,
\item[ii)]$K_0$ is a cycle with at most one vertex in $C$, and if $m$ is the sum of the $q$-values over all faces of $G$ drawn in the open disk bounded by $K_0$, then $\vert m \vert > \vert K_0 \vert$.
\end{itemize}
\end{lemma}

For a multiset of numbers $F$, let $\ell(F)$ denote the smallest integer $\ell$ such that there exists a triangle-free plane graph $G$ with outer face bounded by an $\ell$-cycle $C$, such that $G$ is $C$-critical and $S(G)=F$.
It is known from~\cite{trfree4} that $\ell(\{i\}) = i+2$ and $\ell(\{5,6\}) = 9$.

%
%
% Lemma 1
%
%
%

The next lemma from \cite{trfree1} describes interiors of cycles in critical graphs and will be used frequently in this paper.

\begin{lemma}[\cite{trfree1}] \label{lem:inside}
Let $G$ be a plane graph with outer face $C$. Let $K$ be a non-facial cycle in $G$,
 and let $H$ be the subgraph of $G$ drawn in the closed disk bounded by $K$. If $G$ is $C$-critical for $k$-coloring, then $H$ is $K$-critical for $k$-coloring.
\end{lemma}

Next we include several definitions used throughout the rest of the paper. For the definitions, we assume that $G$ is a graph with outer face bounded by a cycle $C$.

An $x, y$-path is a path with endpoints $x$ and $y$. 
Given $a, b, c, d\in V(C)$, let $C(a, b; c, d)$ denote the $a, b$-subpath of $C$ that does not contain vertices $c$ and $d$ as internal vertices. 
An $x, y$-path $K$ is an {\it $(x, y; f)$-cut} if $x, y$ are on $C$, no internal vertices of $K$ are on $C$, and the face $f$ is in the region bounded by $K$ and the clockwise $x, y$-subpath of $C$. 

Let $K_1$ and $K_2$ be two distinct paths with endpoints on $C$ that are internally disjoint from $C$.
For $i \in \{1,2\}$ let $P_i$ be a subpath of $C$ with the same endpoints as $K_i$ and
label the endpoints of $K_i$ by $u_i$ and $v_i$, where $u_i$ is the first vertex of $P_i$
when traversing the cycle formed by $K_i$ and $P_i$ clockwise.
The \emph{order} of $u_1,v_1,u_2,v_2$ is an ordering of these vertices when traversing along $C$ in the clockwise order. If $x_1 \in \{u_1, v_1\}$ and $x_2 \in \{u_2,v_2\}$ are the same vertex, we define the order of $x_1$ and $x_2$ in the following way. Let $x_1=y_0,\ldots,y_m$ be the longest common subpath of $K_1$ and $K_2$. 
We consider the neighbors $N$ of $y_m$ in the counterclockwise order ending with $y_{m-1}$ or a vertex of $C$ if $m=0$. 
If a vertex of $K_1$ appears in $N$ before every vertex of $K_2$, then $x_1$ is before $x_2$ in the ordering, otherwise $x_2$ is before $x_1$.

Every order, or the pair $K_1,K_2$, is assigned a \emph{kind} $(t_1t_2)$, where $t_i$ 
is the number of vertices from $\{u_{3-i},v_{3-i}\}$ that are in $P_i$.
Hence there are only five possible kinds; namely $(00)$, $(02)$, $(20)$, $(22)$, and $(11)$. 
If $K_1$ is the same path as $K_2$, we can pick any order that will give kind $(00)$, $(02)$,  or $(20)$. 

Suppose that the order is $u_1,v_2,v_1,u_2$, which gives kind (11).
By planarity, there exists a vertex $x$ such that $x$ is an internal vertex of both $K_1$ and $K_2$. Denote by  $K_i^y$ a subpath of $K_i$ with endpoints $x$ and $y$ for $i \in \{1,2\}$
and $y \in \{u_i,v_i\}$;  see Figure~\ref{fig-common}.
Let $F_i$ be the set of $5^+$-faces that are in the interior of the cycle bounded by  $P_i$ and $K_i$
and in the exterior of the cycle bounded by $P_{3-i}$ and $K_{3-i}$ for $i \in \{1,2\}$.
The vertex $x$ is a \emph{common point} of $K_1$ and $K_2$ if 
every face in $F_1$ is in an interior face of the subgraph of $G$ induced by $P_1,K_1^{v_1},K_2^{v_2}$
and every face in $F_2$ is in an interior face of the subgraph of $G$ induced by $P_2,K_1^{u_1},K_2^{u_2}$.

\begin{figure}
\begin{center}
 \includegraphics[scale=1]{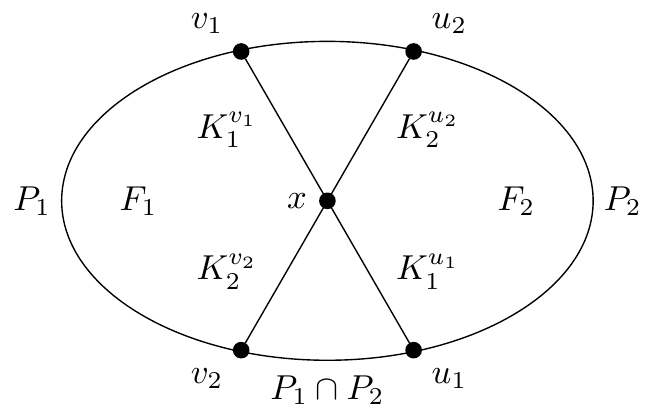}\\
\end{center}
\caption{Kind (11) and a common point $x$.}
\label{fig-common}
\end{figure}

It is possible to show that there always exists a common point for the kind (11) if $K_1$ and $K_2$ are not too long.
\begin{lemma}\label{lem:common}
Let $G$ be a triangle-free plane graph with outer face $C$ where every 4-cycle bounds a face and let $K_1$ and $K_2$ be paths in $G$ with endpoints 
$u_1,v_1,u_2,v_2$ in $C$. Let $K_1$ and $K_2$ be internally disjoint with $C$ and the order of $u_1,v_1,u_2$, and $v_2$  form (11).
Then there exists a common point  for $K_1$ and $K_2$ if either
 $\max \{|K_1|,|K_2|\} \leq 7$ and $\min \{|K_1|,|K_2|\} \leq 6$ 
or 
 $|K_1|=|K_2|=7$ and the endpoints of $K_1$ and $K_2$ are the same.
\end{lemma}
\begin{proof}
We describe operations that eliminate candidates for common points.
Eventually, we show that the situation is equivalent to the case where $K_1$ and $K_2$
share exactly one vertex and then it is easy to see it is the common point.

By planarity and the kind (11), there must be at least one vertex of $G$ that is internal vertex of both $K_1$
and $K_2$. 
When traversing $K_2$ from $v_2$ to $u_2$ we label the internal vertices of $K_2$ that
are also vertices of $K_1$ by $c_1,c_2,c_3,\ldots$. 
These vertices are candidates to be common points. 
We order them by their distance from $u_1$ on the path $K_1$.
Let $P_1$ be the clockwise path in $C$ from $u_1$ to $v_1$.

An edge of $K_2$ is \emph{inside} if it is drawn inside of the open disk bounded by the cycle formed by $P_1$ and $K_1$
or if it is incident with $v_2$.
An edge of $K_2$ is \emph{outside} if it is drawn outside of the closed disk bounded by the cycle formed by $P_1$ and $K_1$
or if it is incident with $u_2$. 
Notice that if an edge of $K_2$ is neither inside nor outside, it is also an edge of $K_1$
and we call it \emph{shared}.

Now we do several modifications to $G$, $K_1$ and $K_2$ such that the $5^+$-faces are not affected but some candidates for common vertices are eliminated. 

For some $i$, if $c_i$ is adjacent to one shared edge and one not shared edge $h$, then we split $c_i$ into two vertices $c^1_i$ and $c^2_i$,
creating a new $4$-face containing $c^1_i,c^2_i$ and the two other neighbors of $c_i$ in $K_1$.
We can replace $c_i$ by $c^1_i$ and $c^2_i$ in $K_1$ and $K_2$ such that 
$c^2_i$ is inside or outside of the new cycle formed by $P_1K_1$ if $h$ is inside or outside, respectively.
By performing this operation, we decrease the number of vertices in the intersection of $K_1$ and $K_2$ and we can assume $K_2$ has no shared edges.

If both edges of $K_2$ incident with $c_i$ for some $i$ are inside (or outside) we split $c_i$ into two vertices $c^1_i$ and $c^2_i$,
creating a new $4$-face containing $c^1_i,c^2_i$ and the other neighbors of $c_i$ in $K_1$.
We can replace $c_i$ by $c^1_i$ and $c^2_i$ in $K_1$ and $K_2$, respectively.  
We label the vertices such that $c^2_i$ is in the interior (or exterior, respectively) of the
cycle bounded by $K_1$ and $P_1$. 
By performing this operation, we decrease the number of vertices in the intersection of $K_1$ and $K_2$ and assume that $c_i$ is incident to one inside edge and one outside edge for all $i$.

If $c_i$ and $c_{i+1}$ are consecutive in the order given by the distance from $u_1$ and 
the subpaths of $K_1$ and $K_2$ with endpoints $c_i$ and $c_{i+1}$ form a 4-cycle $K$ (hence a 4-face), then
we can reroute the paths such that the length of one of the paths is decreased or we create two vertices that are both incident with only inside or only outside edges.
If one of the two paths forming $K$ has length one, the other one has length three and replacing the longer one by an edge decreases the length of $K_1$ or $K_2$ (it also creates a new shared that  we can eliminate). 
If both paths have length two, we swap them and now both $c_i$ and $c_{i+1}$ are incident
to two edges that are both inside or both outside and they can be eliminated.

Notice that these operations do not increases the length of $K_1$ or $K_2$, do not create new vertices in the intersection of $K_1$ or $K_2$, do not affects locations or number of $5^+$-faces of $G$
with respect to regions formed by $K_1C$ and $K_2C$. Hence a common point in the result would be a common point in the original configuration.

With use of computer, we generate all possible patterns where none of the above operations can be applied.
In all of the patters with  $\max \{|K_1|,|K_2|\} \leq 7$ and $\min \{|K_1|,|K_2|\} \leq 6$, there is only one vertex shared by $K_1$ and $K_2$, which is the common point.

If $|K_1| = |K_2| =  7$, there are eight patterns with more than one internal vertex in the intersection of $K_1$ and $K_2$. Four of them do not actually form (11) and none of the three operations was used on them which is a contradiction.
The other four contain 4-cycles that do not bound a face which is also a contradiction.
The program including the eight patterns is available with all the other programs used in this paper.
\end{proof}

\section{Proof of Theorem~\ref{thm:9cycleprecise}}
Let $\Ss_{k}$ be the set of possible multisets of lengths of $5^+$-faces in a connected plane graph of girth at least $4$ where the length of the precolored face is $k$.
The result of Dvo\v{r}\'ak, Kr\'{a}\soft{l}, and Thomas~\cite{trfree4} implies among others that
$\Ss_6 = \{\emptyset\}$, $\Ss_7 = \{\{5\}\}$, $\Ss_8 = \{\emptyset, \{6\},\{5,5\}\}$, and  $\Ss_9 = \{\{7\},\{5\},\{6,5\},\{5,5,5\}\}$.

By the previous paragraph, we have four cases to consider when $C$ has length $9$.
The case of one $7$-face was already resolved by Dvo\v{r}\'ak and Lidick\'{y}~\cite{00DvLi}, and it is described in Theorem~\ref{thm:9cycleprecise}(b). We restate the result from \cite{00DvLi} in the next subsection as Theorem~\ref{thm:n-2}.
We resolve the remaining three cases in Lemmas~\ref{l56a}, \ref{l56b}, \ref{l5a}, \ref{l5}, \ref{lall5a}, and \ref{lall5}
in the following three subsections.
In order to simplify the cases, we first solve the case when $C$ has a chord.

If $G$ is $C$-critical and $C$ has a chord, then
Lemma~\ref{lem:inside} implies
that $G$ can be obtained by identifying two edges of the outer faces of two different smaller critical graphs or cycles.
Lemma~\ref{lem:chord} shows that the converse is also true.

\begin{lemma}\label{lem:chord}
Let $G_i$ be either a cycle $C_i$ or a triangle-free plane $C_i$-critical graph, where $|C_i| \geq 4$ for $i \in \{1,2\}$.
Let $G$ be the graph obtained 
%from $G_1$ and $G_2$ 
by identifying $e_1 \in E(C_1)$ and $e_2 \in E(C_2)$ and let $C$ be the longest cycle formed by $E(C_1) \cup E(C_2)$ after the identification.
 Then $G$ is $C$-critical, where $|C|=|C_1|+|C_2|-2$.
\end{lemma}
\begin{proof}
Let $e\in E(G)\setminus E(C)$.

Suppose first that $e \in E(G_i)-e_i$ for some $i \in \{1,2\}$.
Since $G_i$ is either a cycle or a $C_i$-critical triangle-free plane graph and it contains $e$ that is not on the boundary, $G_i$ is $C_i$-critical. Hence there exists a $3$-coloring $\varphi$
of $C_i$ that extends to a proper $3$-coloring of $G_i-e$ but does
not extend to a proper $3$-coloring of $G_i$. 
Since $G_{3-i}$ is triangle-free, there exists a proper 3-coloring $\varrho$ of $G_{3-i}$ by Gr\"otzsch's Theorem~\cite{grotzsch1959}.
By permuting colors we may assume that $\varphi$ and $\varrho$ agree on $e_i$ and $e_{3-i}$. This gives a proper 3-coloring of $C$ showing that $G$ is $C$-critical with respect to $e$.

The other case is when $e$ is the result of the identification of $e_1$ and $e_2$.
Let $u,v$ be the vertices of $e$. Since $G-e$ is a triangle-free planar graph,
there exists a proper 3-coloring $\varphi$ of $G-e$ such that $\varphi(u)=\varphi(v)$; 
this is a result of Aksenov et al.~\cite{2002AkBoGl} that was simplified by Borodin et al.~\cite{2014BoKoLiYa}.
Let $\varrho$ be the restriction of $\varphi$ to $C$. 
Clearly, $\varrho$ can be extended to a proper 3-coloring of $G-e$ but not to a proper 3-coloring of $G$.
\end{proof}

Therefore, we can enumerate $C$-critical triangle-free plane graphs $G$ where $C$ has a chord and has length $9$ by identifying edges from two smaller graphs with outer faces of lengths either 4 and 7 or 5 and 6. 
Since there are no $C$-critical graphs when $|C|\in\{4, 5\}$, 
we just use a $4$-cycle and a $5$-cycle.
The resulting graphs are depicted in Figure~\ref{fig-all-critical-graphs} (a), (b), (c1),  (c2), (f1), and (f2),
where some of the vertices may be identified.

\subsection{One 7-face}

The case of one $7$-face is solved by a more general result from \cite{00DvLi}. 
The result works for graphs with an outer face of length  $k$ and one internal face of length $k-2$.
Let $r(k)=0$ if $k\equiv 0\pmod{3}$, $r(k)=2$ if $k\equiv 1\pmod{3}$, and $r(k)=1$ if $k\equiv 2\pmod{3}$.

\begin{theorem}[\cite{00DvLi}]\label{thm:n-2}
Let $G$ be a connected triangle-free plane graph with outer face bounded by a $7^+$-cycle $C$ of length $k$.
Suppose that $f$ is an internal face of $G$ of length $k-2$ and that all other internal faces
of $G$ are $4$-faces.  The graph $G$ is $C$-critical if and only if
\begin{itemize}
\item[(a)] $f\cap C$ is a path of length at least $r(k)$ (possibly empty if $r(k)=0$),
\item[(b)] $G$ contains no separating $4$-cycles, and
\item[(c)] for every $(k-1)^-$-cycle $K\neq f$ in $G$, the interior of $K$ does not contain $f$.
\end{itemize}
Furthermore, in a graph satisfying these conditions, a precoloring $\psi$ of $C$ extends to a $3$-coloring of
$G$ if and only if $E(C)\setminus E(f)$ contains both a source edge and a sink edge with respect to $\psi$.
\end{theorem}

In our case, we apply Theorem~\ref{thm:n-2} with $k=9$. Since $r(9)=0$, the $7$-face does not have to share any edges with the outer face. The description is in Theorem~\ref{thm:9cycleprecise}(b) and it is depicted in Figure~\ref{fig-all-critical-graphs}(b).

%%%%%%%%%%%%%%%%%%%%%%%%%%%%%%%%%%%%%%%%%%%%%%%%%%%%%%%%%
%														%
%														%
%														%
%														%
%														%
%														%
%						Case {5,6}						%
%														%
%														%
%														%
%														%
%														%
%														%
%														%
%%%%%%%%%%%%%%%%%%%%%%%%%%%%%%%%%%%%%%%%%%%%%%%%%%%%%%%%%

\subsection{One 5-face and one 6-face}

\begin{lemma}\label{l56a}
Let $G$ be a connected triangle-free plane graph with outer face bounded by a chordless $9$-cycle $C$.
Moreover, let $G$ contain one $5$-face $f_5$ and one $6$-face $f_6$, all other internal faces are $4$-faces, and all non-facial $8^-$-cycles $K$ in $G$ bound $K$-critical subgraphs.
If $\psi$ is a 3-coloring of $C$ that does not extend to a 3-coloring of $G$, then $\psi$ extends to
$G-e$ for every $e \in E(G)\setminus E(C)$.
\end{lemma}
\begin{proof}
Let $e\in E(G)\setminus E(C)$. 
We want to show that $\psi$ extends to a proper $3$-coloring of $G-e$.
Suppose that $\psi$ does not extend to a $3$-coloring of $G-e$. 
Then there exists a $C$-critical subgraph $H$ of $G-e$, such that the $3$-colorings of $C$ that extend to $G-e$ are exactly the $3$-colorings of $C$ that extend to $H$.
Since $H$ is $C$-critical, its multiset of  $5^+$-faces is one of $\{5\}, \{7\}, \{5,6\}, \{5,5,5\}$.
Since all non-facial $8^-$-cycles $K$ in $G$ bound $K$-critical subgraphs, Lemma~\ref{lem:inside} implies that 
every $5$-face of $H$ is a $5$-face of $G$, 
every $7$-face of $H$ contains exactly one $5$-face of $G$, 
and a $6$-face of $H$ contains no $5$-faces in the interior.
Hence, $H$ contains one odd $5^+$-face and one even $6^+$-face or one odd $9^+$-face, and the only option for the multiset of $5^+$-faces of $H$ is $\{5,6\}$.
That would mean that $G$ is the same graph as $H$, and this is a contradiction.
\end{proof}

Notice that Lemma~\ref{l56a} implies that in order to prove $C$-criticality, it is enough to find one
coloring that does not extend. In Figure~\ref{fig-all-critical-graphs} we depict colorings that do not extend.

Now we prove the other direction of the Theorem~\ref{thm:9cycleprecise}. We start by the following lemma that we
prove separately for future reference and then continue with the main part Lemma~\ref{l56b}.

\begin{lemma}\label{clno9}
Let $G$ be a connected triangle-free plane graph with outer face bounded by a chordless $9$-cycle $C$.
Moreover, let $G$ contain one $5$-face $f_5$ and one $6$-face $f_6$ and all other internal faces are $4$-faces.
If  $G$ is $C$-critical $\psi$ is a $3$-coloring of $G$ with 9 source edges then $\psi$ extends to a 3-coloring of $G$.
\end{lemma}
\begin{proof}
Suppose for a contradiction that $\psi$ does not extend to a 3-coloring of $G$.
Hence there is just one $\psi$-balanced layout $q$ with $q(f_5)=-3$ and $q(f_6)=-6$. %, and $c(q,\psi)=9$. 
Let $K_0$ be obtained from Lemma~\ref{lemma5}.

If $K_0$ is a cycle, then Lemma~\ref{lemma5} implies $9 > |K_0|$.
Let $m$ denote the sum of the $q$-values of the faces in the interior of $K_0$.
By Lemma~\ref{lemma5}, $\vert m \vert > k_0$.
If both $f_5,f_6$ are in the interior of $K_0$, then $\vert m \vert = \vert q(f_5)+q(f_6) \vert = 9$, contradicting the fact that $\vert m \vert > k_0$ since $k_0\geq \ell(\{5,6\})=9$. 
If $f_5$ is in the interior of $K_0$, but  $f_6$ is not, then $\vert m \vert = 3$, while $\ell(\{5\})=5$, a contradiction again.  
Similarly, we obtain a contradiction when $f_6$ is in the interior of $K_0$ but $f_5$ is not, since $\ell(\{6\})=6$ and $\vert m \vert \leq 6$. 

Therefore $K_0$ is always a path joining two distinct vertices of $C$.
These endpoints of $K_0$ partition the edges of $C$ into two paths $X$ and $Y$ intersecting at the endpoints of $K_0$.
For $Z \in \{X,Y\}$, recall that $\ns Z$ and $\nt Z$ denotes the number of source edges and sink edges, respectively, among the edges of $Z$ in coloring $\psi$. 
The described structure is shown in Figure~\ref{fig-65}.
Let $R_X$ and $R_Y$ be the subgraph of $G$ induced by vertices in the closed interior of the cycle formed by $K_0,X$ and $K_0,Y$ respectively.

Note that $\ns X+\ns Y=9$ and $\nt X+\nt Y=0$. 
If both $f_5$, $f_6$ belong to $R_X$, then Lemma~\ref{lemma5} implies $9 -n^s_X >   k_0$ and
Lemma~\ref{lem:inside} implies $n^s_X+k_0 \geq 9$ since $\ell(\{5,6\})=9$, which is a contradiction.
By symmetry, $R_Y$ does not contain both $f_5$ and $f_6$.

Without loss of generality, suppose $f_6$ belongs to $R_X$ and $f_5$ belongs to $R_Y$.
Lemma~\ref{lem:inside} implies that $\ns X + k_0 \geq 6$, which gives $k_0 \geq 6 - \ns X$.
Lemma~\ref{lemma5} implies $|\ns X - 6| > k_0$. 
Combining the inequalities give
$|\ns X - 6| >  6 - \ns X$, which implies $\ns X > 6$. Hence $\ns Y < 3$.
Analogously, we obtain  $\ns Y + k_0 \geq 5$ and $|\ns Y - 3| > k_0$, whose combination
gives $3 - \ns Y  > 5 - \ns Y$, which is a contradiction.
\end{proof}

\begin{lemma}\label{l56b}
Let $G$ be a connected triangle-free plane graph with outer face bounded by a chordless $9$-cycle $C$.
Moreover, let $G$ contain one $5$-face $f_5$ and one $6$-face $f_6$ and all other internal faces are $4$-faces.
If  $G$ is $C$-critical, 
then  $G$ is described by Theorem~\ref{thm:9cycleprecise}(c),(d), and is depicted in Figure~\ref{fig-all-critical-graphs}(c1),(c2),(d1), and (d2).
\end{lemma}
\begin{proof}
Since $G$ is $C$-critical, from Lemma~\ref{lem:inside} follows that  every non-facial $8^-$-cycle $K$ in $G$ bounds a $K$-critical subgraph. 
Since $G$ is $C$-critical, there exists a 3-coloring $\psi$ of $C$ that does not extend to a proper 3-coloring of $G$.

By symmetry, we assume that $C$ has more source edges than sink edges. Hence $C$ has either 9 or 6 source
edges. Lemma~\ref{clno9} eliminates the case of 9 source edges hence $C$ has 6 source edges.
Let $q$ be a $\psi$-balanced layout of $G$. 
Let $K_{0} \subset G$ be obtained by Lemma~\ref{lemma5} and let $k_0 = \vert K_0\vert$.

First suppose that $K_0$ is a cycle. Let $m$ denote the sum of the $q$-values of the faces in the interior of $K_0$.
By Lemma~\ref{lemma5}, $\vert m \vert > k_0$.
If both $f_5,f_6$ are in the interior of $K_0$, then $\vert m \vert = \vert q(f_5)+q(f_6) \vert = 6$, contradicting the fact that $\vert m \vert > k_0$ since $k_0\geq \ell(\{5,6\})=9$. 
If $f_5$ is in the interior of $K_0$, but  $f_6$ is not, then $\vert m \vert = 3$, while $\ell(\{5\})=5$, a contradiction again.  
Similarly, we obtain a contradiction when $f_6$ is in the interior of $K_0$ but $f_5$ is not, since $\ell(\{6\})=6$ and $\vert m \vert \leq 6$. 

Therefore $K_0$ is always a path joining two distinct vertices of $C$.
These endpoints of $K_0$ partition the edges of $C$ into two paths $X$ and $Y$ intersecting at the endpoints of $K_0$.
For $Z \in \{X,Y\}$, recall that $\ns Z$ and $\nt Z$ denotes the number of source edges and sink edges, respectively, among the edges of $Z$ in coloring $\psi$. 
The described structure is shown in Figure~\ref{fig-65}.
Let $R_X$ and $R_Y$ be the subgraph of $G$ induced by vertices in the closed interior of the cycle formed by $K_0,X$ and $K_0,Y$ respectively.

\begin{figure}
\begin{center}
 \includegraphics[scale=1]{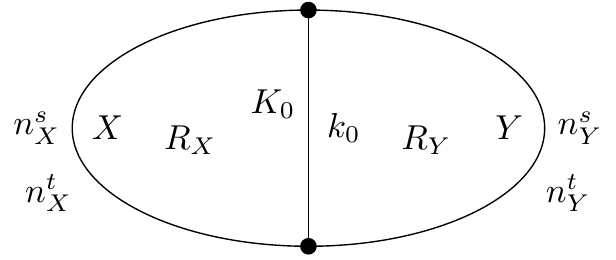}
\end{center}
\caption{The structure of a cut in  $G$. }\label{fig-65}
\end{figure}

\begin{claimn} \label{cl56}
If $q$ is a $\psi$-balanced layout either with $q(f_5)=-3$ and $q(f_6)=0$ or with $q(f_5)=3$ and $q(f_6)=-6$,
then both $R_X$ and $R_Y$ contain exactly one of $f_5$ and $f_6$.
\end{claimn}

\begin{proof}
By Lemma~\ref{clno9}, $C$ contains 6 source edges, hence $\ns X+\ns Y=6$ and $\nt X+\nt Y=3$. 
By symmetry, suppose for a contradiction that both $f_5$, $f_6$ belong to $R_X$.
Notice that $q(f_5)+q(f_6)=-3$ in both layouts.
By Lemma~\ref{lem:inside}, $\ns X+\nt X+k_0\geq \ell(\{5,6\})=9$, and by Lemma~\ref{lemma5}, $|\ns X -3 - \nt X| > k_0$. 

If $\ns X -3 - \nt X > k_0$, then we obtain $\ns X -3 - \nt X > k_0 \geq 9 - \ns X-\nt X$. This gives $\ns X > 6$, which is a contradiction.

If $-\ns X +3 +\nt X > k_0$, then we obtain $-\ns X +3 + \nt X > k_0 \geq 9 - \ns X-\nt X$. This gives $\nt X > 3$, which is a contradiction.
\end{proof}

Since $C$ has 6 source edges, we have two different $\psi$-balanced layouts.
Let $q_1$ and $q_2$ be the layout where $q_1(f_5)=-3$, $q_1(f_6)=0$, and $q_2(f_5)=3$, $q_2(f_6)=-6$, respectively.
Let $K$ and $L$ be the subgraph of $G$ obtained by Lemma~\ref{lemma5} applied to $q_1$ and $q_2$, respectively, and let $k=|K|$  and $l=|L|$. 
Note that we already showed that each of $K$ and $L$ is a path joining pairs of distinct vertices of $C$; let $K$ and $L$ be a $(v_1, v_2; f_5)$-cut and $(w_1, w_2; f_6)$-cut, respectively. 
The paths $K$ and $L$ form a structure of kind (00), (11), (22), (20), or (02), see Figure~\ref{fig-65nocross}, Figure~\ref{fig-65cross}, and Figure~\ref{fig-650220} for illustration. We discuss these cases in separate claims.

\begin{figure}
\begin{center}
 \includegraphics[scale=1]{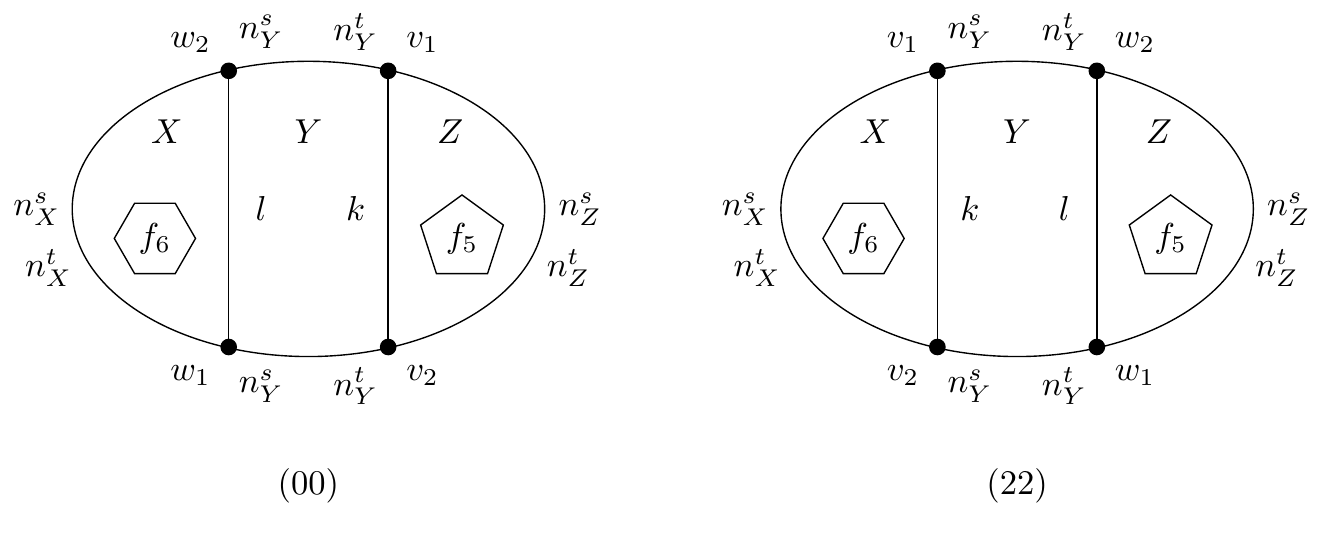}\\
\end{center}
\caption{The cases where $K$ and $L$ are of kinds (00) and (22).}
\label{fig-65nocross}
\end{figure}

\begin{claimn}
 If $K$ and $L$ are of kind (00), then $G$ is depicted in Figure~\ref{fig-all-critical-graphs}(c1).
\end{claimn}
\begin{proof}
Note that $K$, $L$ are not necessarily disjoint.
By symmetry, let $X$ be $C(w_1, w_2; v_1, v_2)$ such that the disk bounded by $L$ and $X$ contains $f_6$.
Similarly, let $Z$ be $C(v_1, v_2; w_1, w_2)$ such that the disk bounded by $K$ and $Z$ contains $f_5$. 
Denote by $Y$ the edges of $C$ that are neither in $X$ nor in $Z$.
See Figure~\ref{fig-65nocross}~(00).

By the assumption that $C$ has no chord, $k \geq 2$ and $l \geq 2$. 
By Claim~\ref{clno9}, we know $\ns X+\ns Y+\ns Z=6$ and $\nt X+\nt Y+\nt Z=3$.

Lemma~\ref{lem:inside} implies $l+\Sum X\geq \ell(\{6\})=6$. 
Moreover, by parity, $l+\Sum X$ must be even.
Similarly, Lemma~\ref{lem:inside} implies that $k+\Sum Z \geq \ell(\{5\}) = 5$ and it is odd.
Lemma~\ref{lemma5} applied to $q_1$ and $q_2$ implies
$|\ns X +\ns Y-\nt X-\nt Y| > k$ and
$|3+\ns Z+\ns Y   -\nt Z-\nt Y| > l$, respectively.

Here is the summary of the constraints:
\begin{align*}
|\ns X +\ns Y-\nt X-\nt Y| &> k \\ 
|3+\ns Z + \ns Y  -\nt Z-\nt Y| &> l  \\
 l+\Sum X&\geq  6   \text{ and even} \\
  k+\Sum Z& \geq  5  \text{ and odd} \\
  \ns X+\ns Y+\ns Z&=6  \\
  \nt X+\nt Y+\nt Z&=3   \\
 \min\{k,l\} &\geq 2
\end{align*}

All integer solutions to these constraints are in the following table:{
\begin{center}
\setlength{\tabcolsep}{8pt}
\renewcommand{\arraystretch}{1.2}
\begin{tabular}{|c|c|c|c|c|c|c|c|}
\hline
$\ns X$ & $\nt X$ & $\ns Y$ & $\nt Y$ & $\ns Z$ & $\nt Z$ & $k$ & $l$ \\
\hline
0&1&5&0&1&2&2&5  \\ \hline
0&1&6&0&0&2&3&5  \\ \hline
1&1&4&0&1&2&2&4  \\ \hline
1&1&5&0&0&2&3&4  \\ \hline
2&1&3&0&1&2&2&3  \\ \hline
2&1&4&0&0&2&3&3  \\ \hline
3&1&2&0&1&2&2&2  \\ \hline
3&1&3&0&0&2&3&2  \\ \hline
\end{tabular}
\end{center}
}

From these eight solutions we obtain the graph depicted in Figure~\ref{fig-all-critical-graphs}(c1),
up to identification of vertices.
\end{proof}

\begin{claimn}
 If $K$ and $L$ are of kind (22), then $G$ is depicted in Figure~\ref{fig-all-critical-graphs}(c2).
\end{claimn}
\begin{proof}
By symmetry, let $X$ be $C(v_1, v_2; w_1, w_2)$ such that the disk bounded by $K$ and $X$ contains $f_6$.
Similarly, let $Z$ be $C(w_1, w_2; v_1, v_2)$ such that the disk bounded by $L$ and $Z$ contains $f_5$. 
Denote by $Y$ the edges of $C$ that are in neither $X$ nor $Z$.
See Figure~\ref{fig-65nocross}~(22).

By the assumption that $C$ has no chord, $k \geq 2$ and $l \geq 2$. 
By Claim~\ref{clno9}, we know $\ns X+\ns Y+\ns Z=6$ and $\nt X+\nt Y+\nt Z=3$.

Lemma~\ref{lem:inside} implies $k+\Sum X\geq \ell(\{6\})=6$. 
Moreover, by parity, $k+\Sum X$ must be even.
Similarly, Lemma~\ref{lem:inside} implies $l+\Sum Z \geq \ell(\{5\}) = 5$ and it is odd.
Lemma~\ref{lemma5} applied to $q_1$ and $q_2$ implies
$|\ns X -\nt X| > k$ and
$|\ns Z + 3  -\nt Z| > l$, respectively.

Here are the constraints:
\begin{align*}
|\ns X -\nt X| &> k \\ 
|\ns Z + 3  -\nt Z| &> l  \\
 k+\Sum X &\geq6   \text{ and even} \\
  l+\Sum Z  &\geq5  \text{ and odd} \\
  \ns X+\ns Y+\ns Z&=6  \\
  \nt X+\nt Y+\nt Z&=3   \\
 \min\{k,l\} &\geq 2
\end{align*}
All integer solutions to these constraints are in the following table:

\vspace{3mm}

{
\begin{center}
\setlength{\tabcolsep}{8pt}
\renewcommand{\arraystretch}{1.2}
\begin{tabular}{|c|c|c|c|c|c|c|c|}
\hline
$\ns X$ & $\nt X$ & $\ns Y$ & $\nt Y$ & $\ns Z$ & $\nt Z$ & $k$ & $l$ \\
\hline
4&0&0&2&2&1&2&2 \\
\hline
4&0&0&3&2&0&2&3 \\
\hline
\end{tabular}
\end{center}
}

From these two solutions we obtain the graph depicted in Figure~\ref{fig-all-critical-graphs}(c2).
\end{proof}

\vspace{3mm}

\begin{claimn}
 If $K$ and $L$ are of kind (11), then $G$ is depicted in Figure~\ref{fig-all-critical-graphs}(d1) or (d2).
\end{claimn}
\begin{proof}
Assume that $K$ and $L$ are of kind (11), so that the clockwise order of their endpoints on $C$ is $v_1,w_2,v_2,w_1$. 
Let $v_1,w_2,v_2,w_1$ partition $C$ into four paths $X,Y,Z,W$ in the clockwise order such
that $X$ is an $w_1, v_1$-path. 
Moreover, the disk bounded by $X,Y,L$ contains $f_6$ and the disk bounded by $K,Y,Z$ contains $f_5$.
See Figure~\ref{fig-65cross} for an illustration. 

First we show that $|K| \leq 6$ and $|L| \leq 7$.
We obtain the following set of constraints by applying  Lemma~\ref{lemma5} and Lemma~\ref{lem:inside}.
\begin{align*}
| \ns X + \ns W - \nt X - \nt W | > |K|  \\ 
| \ns Z + \ns W +3 - \nt Z - \nt W | > |L|   \\ 
\Sum X + \Sum Y + |L| &\geq \ell(\{6\})=6 \text{ and even}  
\end{align*}
In all solutions, t $|K| \leq 6$ and $|L| \leq 7$. Hence Lemma~\ref{lem:common} applies
and $K$ and $L$ have a common point $v$.

Partition $L$ into paths $L_1$ and $L_2$ such that 
$L_1$ and $L_2$ is a $v, w_2$-path and a $v, w_1$-path, respectively. 
Do a similar partition of $K$ into $K_1$ and $K_2$.
Since $v$ is a common point, $f_6$ and $f_5$ is contained in interior faces of subgraphs of $G$ induced by $X,K_1,L_2$ and $Z,L_1,K_2$, respectively.
Let $k_i = |K_i|$ and $l_i = |L_i|$ for $i \in \{1,2\}$.

\begin{figure}
\begin{center}
 \includegraphics{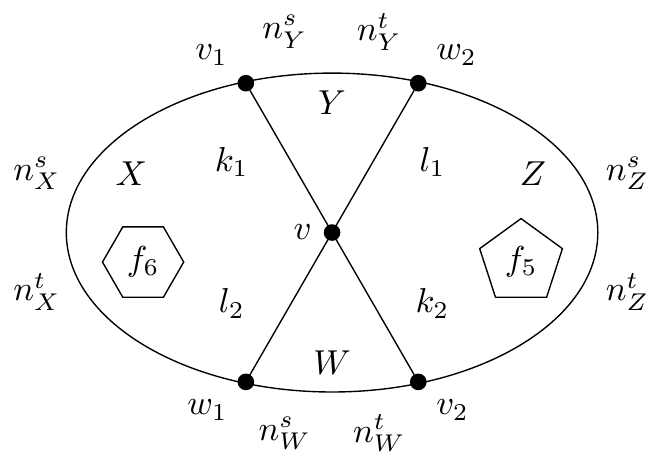}\\
\end{center}
\caption{Case where $K$ and $L$ are of kind (11).}
\label{fig-65cross}
\end{figure}

Note that $\min\{k_1,k_2,l_1,l_2\} \geq 1$ since $v$ is an internal vertex.

We obtain the following set of constraints by applying  Lemma~\ref{lemma5} and Lemma~\ref{lem:inside}.
\begin{align}
| \ns X + \ns W - \nt X - \nt W | > k_1+k_2   \label{eq-65-k1}\\ 
| \ns Z + \ns W +3 - \nt Z - \nt W | > l_1 + l_2  \label{eq-65-k2}\\
k_1+l_2+\Sum X &\geq \ell(\{6\})=6 \text{ and even} \label{eq-65-sz1}\\
l_1+k_2+\Sum Z &\geq \ell(\{5\})=5 \text{ and odd}\\
l_2+k_2+\Sum X+\Sum Y+\Sum Z &\geq \ell(\{5,6\})=9 \text{ and odd}\label{eq-65-sz3}\\
\ns X+\ns Y+\ns Z + \ns W&=6\\
\nt X+\nt Y+\nt Z + \nt W&=3
\end{align}
Inequalities (\ref{eq-65-k1}) and (\ref{eq-65-k2}) come from Lemma~\ref{lemma5}.
Inequalities (\ref{eq-65-sz1})--(\ref{eq-65-sz3}) come from the fact that interiors of cycles
are also critical graphs. 

This system of equations has 68 solutions. In all of them, $\ns X + \nt X + k_1 + l_2  =6$ and $\ns Z + \nt Z + k_2+l_1 = 5$.
Hence the region bounded by $X,K_1,L_2$ is a 6-face and the region bounded by $Z,L_1,K_2$ is a 5-face. In order to generate only general solutions, where faces share as little with $C$ as possible, we add constraints $\ns X + \nt X  = 0$ and $\ns Z + \nt Z = 0$. Then the system has only two solutions. 

{
\begin{center}
\setlength{\tabcolsep}{8pt}
\renewcommand{\arraystretch}{1.2}
\begin{tabular}{|c|c|c|c|c|c|c|c|c|c|c|c|}
\hline
$\ns X$ & $\nt X$ & $\ns Y$ & $\nt Y$ & $\ns Z$ & $\nt Z$ & $\ns W$ & $\nt W$ & $k_1$ & $k_2$ & $l_1$ & $l_2$ \\
\hline
0&0&0&3&0&0&6&0 &1&3&2&5  \\
\hline
0&0&0&3&0&0&6&0 &2&2&3&4  \\
\hline
\end{tabular}
\end{center}
}

From these solutions we obtain graphs depicted in Figure~\ref{fig-all-critical-graphs}(d1) and (d2). We also checked that the 68 solutions
can indeed be obtained from these two by identifying some vertices.
The solutions were obtained by a computer program that is available on arXiv and at \ourURL.
\end{proof}

\begin{figure}
\begin{center}
 \includegraphics[scale=1]{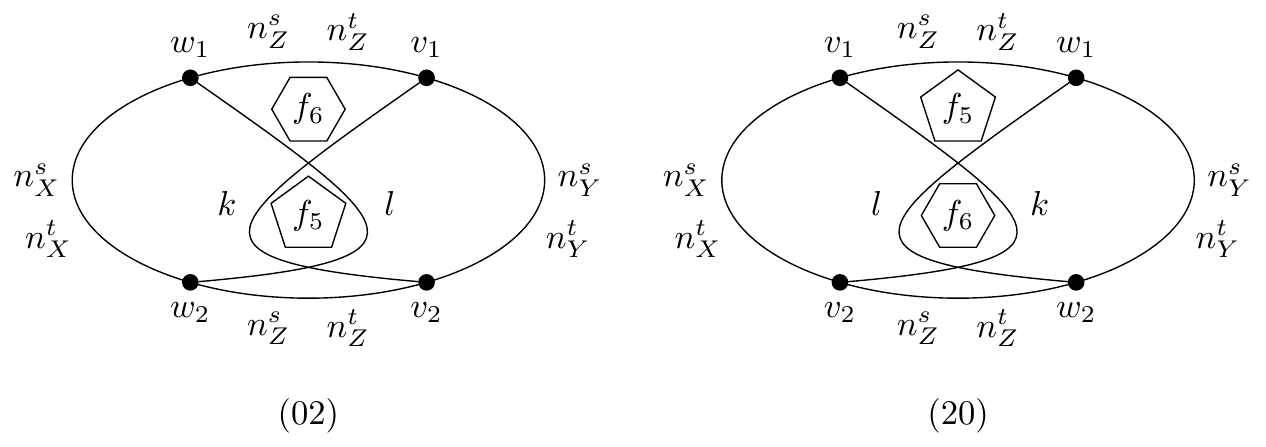}\\
\end{center}
\caption{Case where $K$ and $L$ are of kinds (02) and (20).}
\label{fig-650220}
\end{figure}

\begin{claimn}
 The case where $K$ and $L$ are of kind (02) does not occur.
\end{claimn}
\begin{proof}
Assume that $K$ and $L$ are of  kind (02), so that the clockwise order of their endpoints on $C$ is $v_1,v_2,w_2,w_1$. 
Let $Y$ and $X$ be the clockwise $v_1, v_2$-subpath and $w_2,w_1$-subpath, respectively, of $C$.
Let $Z$ be the edges of $C$ that are in neither $X$ nor $Y$.
See Figure~\ref{fig-650220}~(02) for an illustration. 

Observe that (by the structure of $K$ and $L$) the subgraph of $G$ formed by $Z$, $K$, and $L$
contains in the internal faces both $f_5$ and $f_6$ and at least one additional 4-face.
Hence $k+l+|Z| \geq 15$.
We obtain the following set of constraints by applying  Lemma~\ref{lemma5} and Lemma~\ref{lem:inside}.
\begin{align*}
| \ns X - \nt X | > k  \\ 
| \ns Y -6  -\nt Y  | > l \\
k +l +\Sum Z &\geq 15 \\
\ns X+\ns Y+\ns Z&=6\\
\nt X+\nt Y+\nt Z&=3
\end{align*}
This set of equations has no solution.
\end{proof}

\begin{claimn}
 The case where $K$ and $L$ are of kind (20) does not occur.
\end{claimn}
\begin{proof}
Assume that $K$ and $L$ are of  kind (20), so that the clockwise order of their endpoints on $C$ is $w_1,w_2,v_2,v_1$. 
Let $Y$ and $X$ be the clockwise $w_1, w_2$-subpath and $v_2,v_1$-subpath, respectively, of $C$.
Let $Z$ be the edges of $C$ that are in neither $X$ nor $Y$.
See Figure~\ref{fig-650220}~(20) for an illustration. 

Observe that (by the structure of $K$ and $L$) the subgraph of $G$ formed by $Z$, $K$, and $L$
contains in the internal faces both $f_5$ and $f_6$ and at least one additional 4-face.
Hence $k+l+|Z| \geq 15$.
We obtain the following set of constraints by applying  Lemma~\ref{lemma5} and Lemma~\ref{lem:inside}.
\begin{align*}
| \ns Y - 3 - \nt Y | > k  \\ 
| \ns X +3 -\nt X  | > l \\
k +l +\Sum Z &\geq 15 \\
\ns X+\ns Y+\ns Z&=6\\
\nt X+\nt Y+\nt Z&=3
\end{align*}
This set of equations has no solution.
\end{proof}

This finishes the proof of Lemma~\ref{l56b}.
\end{proof}

%%%%%%%%%%%%%%%%%%%%%%%%%%%%%%%%%%%%%%%%%%%%%%%%%%%%%%%%%
%														%
%														%
%														%
%														%
%														%
%														%
%						Case {5}						%
%														%
%														%
%														%
%														%
%														%
%														%
%														%
%%%%%%%%%%%%%%%%%%%%%%%%%%%%%%%%%%%%%%%%%%%%%%%%%%%%%%%%%

\subsection{One 5-face}

\begin{lemma}\label{l5a}
Let $G$ be a connected triangle-free plane graph with outer face  bounded by a chordless $9$-cycle $C$.
Moreover, let $G$ contain one 5-face $f_5$ that shares a path of length at least two with $C$, all other internal faces of $G$ are $4$-faces, 
and all non-facial $8^-$-cycles $K$ in $G$ bound a $K$-critical graph.
Then $G$ is $C$-critical.
\end{lemma}
\begin{proof}
Let $e \in E(G) \setminus E(C)$.  We want to find a $3$-coloring $\psi$ of $C$ that does not extend to a proper $3$-coloring of $G$ but does extends to a proper $3$-coloring of $G-e$. 
Note that if $e \not\in E(f_5)$, then  $G-e$ has a $5$-face and a $6$-face, and if $e \in E(f_5)$, then $G-e$ has a $7$-face. 

If every coloring of $C$ extends to $G-e$, then we can let $\psi$ be a coloring with 9 source edges since $\psi$ does not extend to $G$ as there is no $\psi$-balanced layout for $G$.
If not all colorings of $C$ extend to $G-e$, then there is a $C$-critical subgraph $H$ of $G-e$
where the same set of precolorings of $C$ extends to $G-e$ as well as to $H$.
The property that every $8^-$-cycle $K$ either bounds a face or a $K$-critical subgraph gives
that $H$ contains either a $5$-face and a $6$-face or a $7$-face.

\begin{itemize}
\item[{\sl Case 1:}] $H$ contains a $5$-face and a $6$-face.\\ 
Let $\psi$ be a $3$-coloring of $C$ containing 9 source
edges; in other words, the colors of the vertices around $C$ are $1,2,3,1,2,3,1,2,3$.
Then $\psi$ extends to a $3$-coloring of $H$ by Claim~\ref{clno9}.
However, $\psi$ does not extend to a $3$-coloring of $G$ since it is not
possible to create a $\psi$-balanced layout for $G$.

\item[{\sl Case 2:}] $H$ contains a $7$-face $f_7$.\\
By Theorem~\ref{thm:n-2}, if $\psi$ is a $3$-coloring of $C$ containing 9 source edges,
then $\psi$ does not extend to a proper $3$-coloring of $H$, and if $\psi$ is a $3$-coloring of $C$ containing 6 source edges and 3 sink edges, then $\psi$ extends to a proper $3$-coloring of $H$
if $E(C)\setminus E(f_7)$ contains both  a sink edge and a source edge with respect to  $\psi$.
Now it remains to observe that there exists a coloring  $\psi$ of $C$ such that $E(f_5) \cap E(C)$ contains two sink edges and the third sink edge is in  $E(C)\setminus E(f_7)$. 
The other edges of $C$ are source edges. 
Such a coloring does not extend to $G$ but it does extend to $G-e$.
\end{itemize}
\end{proof}

\begin{lemma}\label{l5}
Let $G$ be a connected triangle-free plane graph with outer face  bounded by a chordless $9$-cycle.
Moreover, let $G$ contain one 5-face $f_5$ and all other internal faces of $G$ are $4$-faces.
If $G$ is $C$-critical, then $G$ is described by Theorem~\ref{thm:9cycleprecise}(a) and is depicted in Figure~\ref{fig-all-critical-graphs}(a).
\end{lemma}
\begin{proof}
Let $G$ be $C$-critical. 
By Lemma~\ref{lem:inside}, every $8^-$-cycle $K$ bounds a face or a $K$-critical subgraph in $G$. 
Let $e \in E(G) \setminus E(C)$ such that $G-e$ contains a 7-face $f_7$.
Let $\psi$ be a $3$-coloring of $C$ that extends to $G-e$ but does not extend to $G$.

By Theorem~\ref{thm:n-2}, if $\psi$ is a $3$-coloring of $C$ containing 9 source edges,
then $\psi$ does not extend to a proper $3$-coloring of $G-e$.
Hence $\psi$ is a $3$-coloring of $C$ containing 6 source edges and 3 sink edges.
Let $q$ be a $\psi$-balanced layout of $G$. The only possibility is $q(f_5)=-3$.

Since $\psi$ does not extend to $G$ and $q$ is  $\psi$-balanced layout of $G$, Lemma~\ref{lemma5} can be applied.
Notice that Lemma~\ref{lemma5} cannot give that $K_0$ is a cycle since $|m| \leq 3$ and
there is no cycle of length at most 2. Hence $K_0$ is a path, and let $k_0=|K_0|$. 

Let the endpoints of $K_0$ partition $C$ into two paths $X$ and $Y$ that are internally disjoint
and have the same endpoints as $K_0$.
Since $\psi$ has six source edges, we obtain $\ns X+\ns Y=6$ and $\nt X+\nt Y=3$. 
By symmetry assume that $f_5$ is in the region bounded by $Y$ and $K_0$.
Lemma~\ref{lemma5} implies $|\ns Y -3 - \nt Y| > k_0$.
Since $Y$ contains $f_{5}$, $k_{0}+\Sum Y\geq\ell(\{5\})=5$ and odd. 
Because $C$ has no chords, $k_0 \geq 2$.
We solve this system of constraints by a computer program. 
The solutions are in Table~\ref{tab:soll5}. 
Sketches of the solutions are in Figure~\ref{fig-5-alone}.

\begin{table}[ht]
\begin{center}
\setlength{\tabcolsep}{8pt}
\renewcommand{\arraystretch}{1.2}
\begin{tabular}{|c|c|c|c|c|c|}
\hline
\# &$\ns X$ & $\nt X$ & $\ns Y$ & $\nt Y$ & $k_{0}$\\
\hline
(a)& 6&1&0&2&3 \\
\hline
(b)& 6&0&0&3&2 \\
\hline
(c)&5&1&1&2&2 \\ 
\hline
(d)&6&0&0&3&4 \\
\hline
(e)&5&0&1&3&3 \\
\hline
(f)&4&0&2&3&2 \\
\hline
\end{tabular}
\end{center}
\caption{Solutions in Lemma~\ref{l5}}\label{tab:soll5}
\end{table}

\def\tempL{0.17}
\begin{figure}
\begin{center}
\includegraphics[width=\tempL\textwidth,page=1]{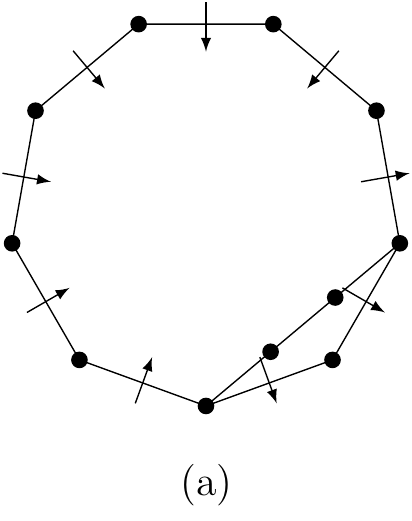}
\includegraphics[width=\tempL\textwidth,page=2]{fig-5-alone}
\includegraphics[width=\tempL\textwidth,page=3]{fig-5-alone}

\includegraphics[width=\tempL\textwidth,page=4]{fig-5-alone}
\includegraphics[width=\tempL\textwidth,page=5]{fig-5-alone}
\includegraphics[width=\tempL\textwidth,page=6]{fig-5-alone}
\end{center}
\caption{Solutions in Lemma~\ref{l5}.}\label{fig-5-alone}
\end{figure}
\def\tempL{0.19}

From the first three solutions we obtain that $Y$ is part of a $5$-face $f_{5}$ sharing at least two sink edges with $C$. This is the desired conclusion.

The other three solutions give that $Y,K_0$ form a  $7$-cycle sharing at least three sink edges with $C$.
We need to rule out this case.  
Since the cycle formed by $Y,K_0$ does not bound  a face in $G$, it must contain an edge $e'$ in its interior.
Since $G$ is $C$-critical, there exists a proper $3$-coloring $\varrho$ of $C$ that does not extend to $G$
but does extend to $G-e'$. 
Notice that the solutions (d), (e), and (f)  also describe all patterns of a $3$-coloring of
$C$ that do not extend to $G$, in particular for $\varrho$.
In all three cases, $X$ contains only source edges and $|X| > |K_0|$.
Since the cycle formed by $X,K_0$ contains only 4-faces in its interior, it is not possible to create a $\rho$-balanced layout in its interior.
Hence $\varrho$ does not extend to subgraph of $G-e'$ bounded $X,K_0$ . 
Therefore  $\varrho$ does not extend $G$. This contradicts the $C$-criticality of $G$.
Hence the cases  (d), (e), and (f) do not correspond to $C$-critical graphs.

This finishes the proof of Lemma~\ref{l5}.
\end{proof}

%%%%%%%%%%%%%%%%%%%%%%%%%%%%%%%%%%%%%%%%%%%%%%%%%%%%%%%%%
%														%
%														%
%														%
%														%
%														%
%														%
%						Case {5,5,5}				    %
%														%
%														%
%														%
%														%
%														%
%														%
%														%
%%%%%%%%%%%%%%%%%%%%%%%%%%%%%%%%%%%%%%%%%%%%%%%%%%%%%%%%%

\subsection{Three 5-faces}

\begin{lemma}\label{lall5a}
Let $G$ be a connected triangle-free plane graph with outer face  bounded by a chordless $9$-cycle $C$.
Moreover, let $G$ contain three 5-faces, all other internal faces are $4$-faces, and
all non-facial $8^-$-cycles $K$ in $G$ bound a $K$-critical graph.
If there is a proper 3-coloring $\psi$ of $C$ that does not extend to  a 3-coloring of $G$, then $G$ is $C$-critical.
\end{lemma}
\begin{proof}
Let $e\in E(G)\setminus E(C)$. 
We want to show that $\psi$ extends to a proper $3$-coloring of $G-e$.
Suppose that $\psi$ does not extend to a $3$-coloring of $G-e$. 
Then there exists a $C$-critical subgraph $H$ of $G-e$, such that the $3$-colorings of $C$ that extend to $G-e$ are exactly the $3$-colorings of $C$ that extend to $H$.
Since $H$ is $C$-critical, its multiset of  $5^+$-faces is one of $\{5\}, \{7\}, \{5,6\}, \{5,5,5\}$.
Since all non-facial $8^-$-cycles $K$ in $G$ bound $K$-critical subgraphs, Lemma~\ref{lem:inside} implies that 
every $5$-face of $H$ is a $5$-face of $G$, 
every $7$-face of $H$ contains exactly one $5$-face of $G$, 
and a $6$-face of $H$ contains no $5$-faces in the interior.
Hence, $H$ contains three odd faces, and the only option for the multiset of $5^+$-faces of $H$ is $\{5,5,5\}$.
That would mean that $G$ is the same graph as $H$, and this is a contradiction.
\end{proof}

\begin{lemma}\label{lall5}
Let $G$ be a connected triangle-free plane graph with outer face bounded by a chordless $9$-cycle $C$. 
Moreover, let $G$ contain three 5-faces and let all other internal faces of $G$ be 4-faces.
If $G$ is $C$-critical, then $G$ is described by Theorem~\ref{thm:9cycleprecise}(e)
 and is depicted in Figure~\ref{fig-all-critical-graphs}(Bij) for some $i$ and $j$.
\end{lemma}

\begin{proof}
Let $G$ be a $C$-critical graph containing three $5$-faces.
Hence there is a proper 3-coloring $\psi$ of $C$ that does not extend to a proper 3-coloring of $G$. 
Without loss of generality, assume $C$ has more source edges than sink edges in the coloring $\psi$.
Either $C$ contains 9 source edges and no sink edges or $C$ contains $6$ source edges and $3$ sink edges. 

Given $i\in\{0, 1, 2, 3\}$, let $\ell_5(i)=\ell(S)$ where $S$ is a multiset of cardinality $i$ containing only elements 5. 
Observe that $\ell_5(0)=4, \ell_5(1)=5, \ell_5(2)=8$, and $\ell_5(3)=9$. 

\begin{claimn}\label{cl555-no9}
There are 6 source edges in $C$.
\end{claimn}
\begin{proof}
Suppose for a contradiction that there are 9 source edges.
Hence there is just one $\psi$-balanced layout $q$ assigning $-3$ to every $5$-face.
Let $K_0$ and $m$ be obtained from Lemma~\ref{lemma5}, which says $|m|>|K_0|$.
Let $k=|K_0|$.

Suppose $K_0$ is a cycle. When $i$ of the $5$-faces are in the interior of $K_0$, then $3i=|m|>k\geq \ell_5(i)$, which is a contradiction for all $i\in\{0, 1, 2, 3\}$. 

Therefore, $K_0$ is a path. Let $C$ be partitioned into paths $X$ and $Y$ that both have the same endpoints as $K_0$. 
Note that $\ns X+\ns Y=9$ and $\nt X+\nt Y=0$, which implies $\nt X=\nt Y=0$. 
Since $C$ is chordless, $k \geq 2$.
By symmetry assume that $X,K_0$ form a cycle that has $i\in\{0, 1\}$ of the three 5-faces in its interior.
Lemma~\ref{lemma5} implies that $|\ns X -3i| > k$ and $\ns Y + k \geq  \ell_5(3-i)$.
This set of equations gives a contradiction for all $i\in\{0, 1\}$.
\end{proof}

Hence $C$ contains 6 source edges and 3 sink edges. 
Let $q$ be a $\psi$-balanced layout, and we know that the three $5$-faces of $G$ are assigned $q$-values $3, -3, -3$. Notice there are three different $\psi$-balanced layouts. Let $K_0$ be obtained from Lemma~\ref{lemma5}.

%%%%%%%%%%%%%%% NEW %%%%%%%%%%%%%%%%%%%
\begin{claimn}\label{cl555-AorB-1}
$K_0$ is a path with both endpoints in $C$. 
\end{claimn}
\begin{proof}
Suppose for a contradiction that $K_0$ is a cycle.
Denote by $m$ the sum of the $q$-values of the faces in the interior of $K_0$.
Lemma~\ref{lemma5} implies that $|m|>|K_0|$.
When $i$ of the $5$-faces are in the interior of $K_0$, then $3i\geq |m|>|K_0|\geq \ell_5(i)$, which is a contradiction for all $i\in\{0, 1, 2, 3\}$.
\end{proof}

Claim~\ref{cl555-AorB-1} says that $K_0$ is a path.
Let $C$ be partitioned into paths $X$ and $Y$ that both have the same endpoints as $K_0$. 
Denote by $R_X$ and $R_Y$ the induced subgraph of $G$ whose outer face is bounded by $K_0, X$ and $K_0, Y$, respectively.

\begin{claimn}\label{cl555-AorB-2}
Each $R_X$ and $R_Y$ contains at least one 5-face.
\end{claimn}
\begin{proof}
Note that $\ns X+\ns Y=6$ and $\nt X+\nt Y=3$.
Without loss of generality, $R_X$ contains three $5$-faces.
Hence $\ns X+\nt X + k\geq 9$, and Lemma~\ref{lemma5} gives $|\ns X -3 -\nt X| > k$. This set of constraints has no solution which is a contradiction.
\end{proof}

%%%%%%%%%%%%%%%%%%%%%%%%%%%%%%%%%%%%%%%%%%%%%%%%%%
\begin{figure}[ht]
\centering
\includegraphics{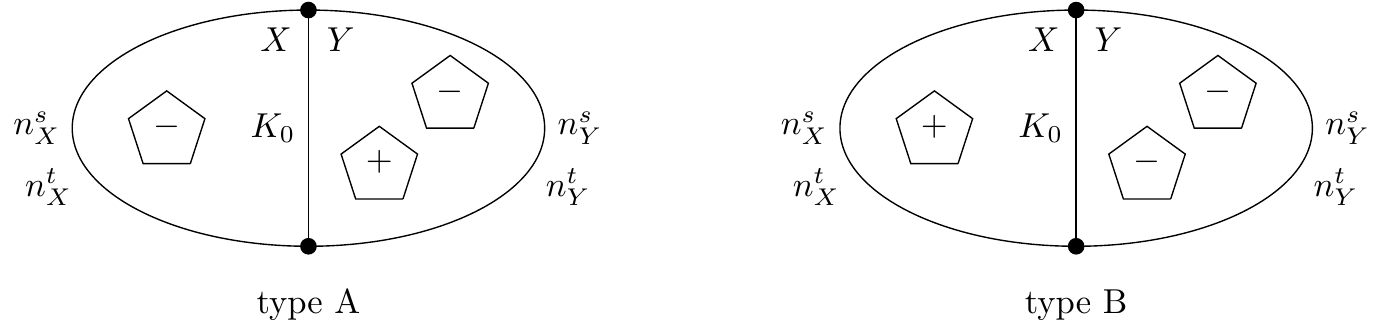}
\caption{When $C$ has 6 source edges and has three $5$-faces. The  possible configurations are of type A (left) and type B (right).}
\label{fig-555-flow3-cuts}
\end{figure}

By Claim~\ref{cl555-AorB-2} and by symmetry, we may assume that $R_X$ contains exactly one 5-face $f$; we call $f$ \emph{lonely} with respect to $K_0$.
If $q(f) = -3$, then we call this configuration \emph{type A} and if $q(f) = 3$, then we call it \emph{type B}; see Figure~\ref{fig-555-flow3-cuts}.

Denote the three different $\psi$-balanced layouts by $q_1$, $q_2$, and $q_3$.
For $i \in \{1,2,3\}$, let $K_i$ be $K_0$ obtained from Lemma~\ref{lemma5} when applied to $q_i$.
By Claim~\ref{cl555-AorB-2}, we can define $f_i$ to be the lonely face for $K_i$.
Notice that $f_1$, $f_2$, and $f_3$ are not necessarily pairwise distinct faces.
Label the endpoints of $K_i$ by $u_i$ and $v_i$ such that $K_i$ is a $(u_i,v_i,f_i)$-cut.
Define $k,l,m$ to be the length of $K_1, K_2, K_3$, respectively.

First we show that configurations of type A do not exist.

\begin{claimn}\label{cl-AA}
Let $q_1$ be a configuration of type A and let $q_2$ be a layout where $q_2(f_1)=3$. Then $q_2$ is not a configuration of type A.
\end{claimn}
\begin{proof}
Suppose for a contradiction that both $q_1$ and $q_2$ give a configuration of type A,  so $q_2(f_1)=3$ and $q_2(f_2)=-3$.
Since  $q_2(f_1)=3$, and $q_2(f_2)=-3$, we have that $f_1$ and $f_2$ are distinct.
Let $f_0$ be the third 5-face.

By symmetry, paths $K_1$ and $K_2$ give one of  four possible kinds (11), (00), (22),  and (20).
The kind (02) is symmetric with (20).
For an illustration, see Figure~\ref{fig-555_cuts_AA}.

Suppose $K_1$ and $K_2$ are of kind (11). 
The situation is depicted in Figure~\ref{fig-555_cuts_AA}~(AA11).
Let $X,A,Y,Z$ be $C(u_2, u_1; v_2, v_1),C(u_1,v_2;v_1,u_2),C(v_2, v_1; u_2, u_1),C(v_1, u_2; u_1, v_2)$ respectively.
We obtain the following constrains that must be satisfied by using Lemma~\ref{lemma5} and Lemma~\ref{lem:inside}.
\begin{align}
|\ns X + \ns Z - \nt X -\nt Z| & > k_1+k_2 \label{555-aa-left-flow1}\\
|\ns Y + \ns Z - \nt X -\nt Y| & > l_1+l_2  \label{555-aa-left-flow2}\\
\Sum X+\Sum A + l &\geq 7 \text{ and odd}  \label{555-aa-7l}\\
\Sum Y+\Sum A + k & \geq7 \text{ and odd}  \label{555-aa-7k}\\
\Sum X+\Sum Y + k + l &\geq 10 \label{555-aa-10}
\end{align}
Inequalities~\eqref{555-aa-left-flow1} and \eqref{555-aa-left-flow2} follow from Lemma~\ref{lemma5}. Inequalities~\eqref{555-aa-7l},  \eqref{555-aa-7k}, and \eqref{555-aa-10} follow from Lemma~\ref{lem:inside} and the structure of the (11) kind.

Suppose $K_1$ and $K_2$ are of kind (00).
The situation is depicted in Figure~\ref{fig-555_cuts_AA}~(AA00).
Let $X$ and $Y$ be $C(u_2, v_2; u_1, v_1)$ and $C(u_1,v_1;u_2,v_2)$ respectively.
Let $Z$ be edges of $C$ that are in neither $X$ nor $Y$.

\begin{figure}[ht]
\centering
\includegraphics[width=0.9\textwidth]{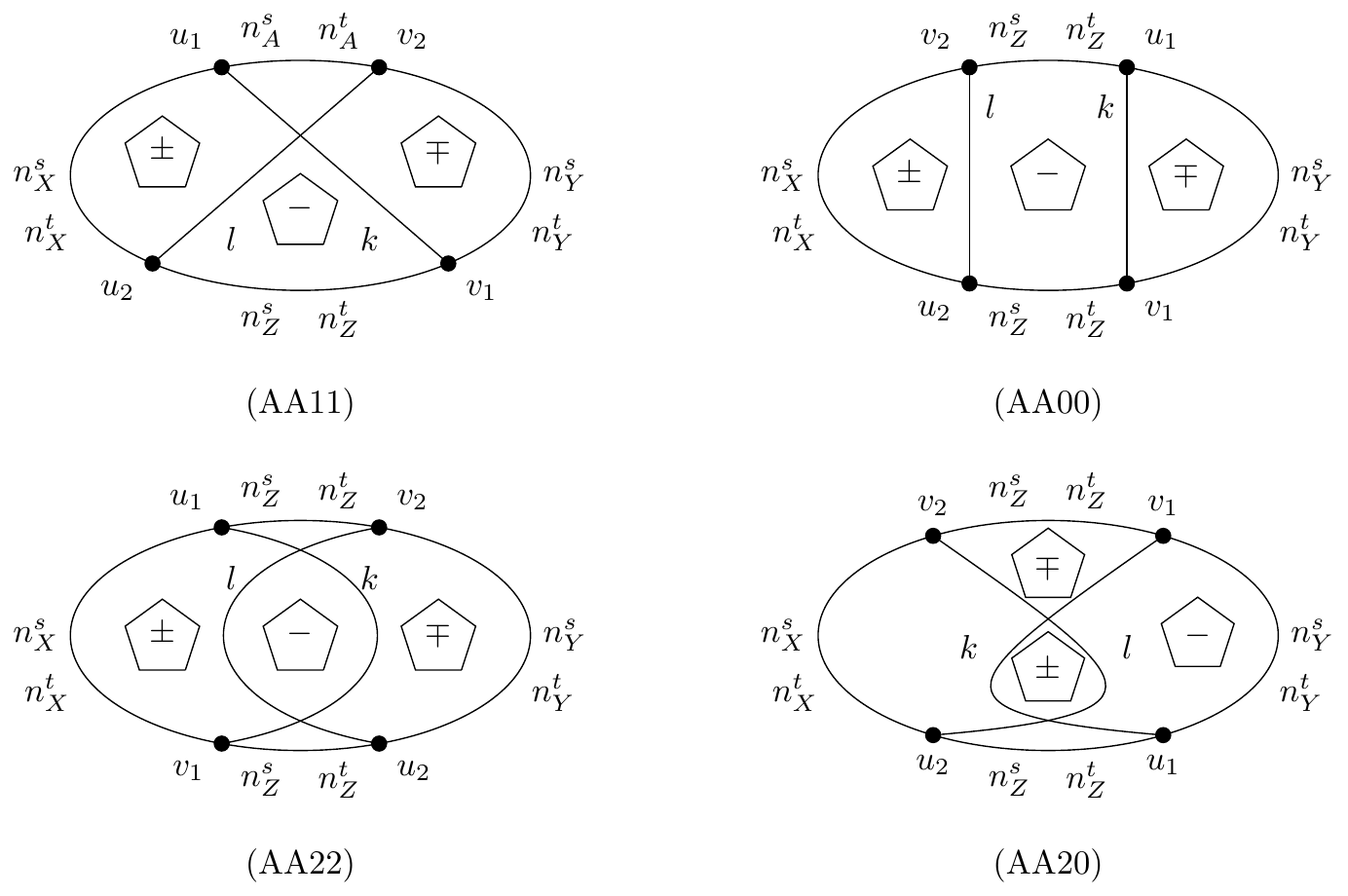}
\caption{Four different cases of two types A. Face $f_1$, $f_2$, and $f_3$ has symbol $\mp$, $\pm$, and $-$, respectively.}
\label{fig-555_cuts_AA}
\end{figure}

As in the previous case we obtain the following set of constraints that must be satisfied.
\begin{align}
 |\ns X + \ns Z - \nt X -\nt Z| & > k \label{555-aa-right-flow1}\\
 |\ns Y + \ns Z - \nt X -\nt Y| & > l  \label{555-aa-right-flow2}\\
 \Sum Y+k & \geq 5 \text{ and odd}\\
 \Sum X+l & \geq 5 \text{ and odd}
\end{align}
Inequalities~\eqref{555-aa-right-flow1} and~\eqref{555-aa-right-flow2} are obtained from Lemma~\ref{lemma5}.
The other inequalities come from Lemma~\ref{lem:inside}.
Recall that we assumed that $C$ has no chords, so we also include that $\min\{k,l\} \geq 2$.
The above set of constraints has no solution. 
Hence $K_1$ and $K_2$ cannot be of kind (00).

The next case (22) is depicted in Figure~\ref{fig-555_cuts_AA} (AA22).
Let $X$ and $Y$ be $C(v_1,u_1;v_1,u_2)$ and $C(v_1, u_2; v_1, u_1)$, respectively.
Let $Z$ be edges of $C$ that are in neither $X$ nor $Y$.

Using Lemmas~\ref{lemma5} and \ref{lem:inside} we obtain the following set of constraints that must be satisfied:
\begin{align*}
 |\ns X + \ns Z - \nt X -\nt Z| & > k \\
 |\ns Y + \ns Z - \nt X -\nt Y| & > l  \\
\Sum Y+  l & \geq 8 \text{ and even}\\
\Sum X+  k & \geq 8 \text{ and even}
\end{align*}
This system has no solution. This finishes the case (22) of Claim~\ref{cl-AA}.

The last case (20) is depicted in Figure~\ref{fig-555_cuts_AA}~(AA20).
Let $X$ and $Y$ be $C(u_2,v_2;v_1,u_1)$ and $C(v_1, u_1; u_2, v_2)$, respectively.
Let $Z$ be edges of $C$ that are in neither $X$ nor $Y$.
Using Lemmas~\ref{lemma5} and \ref{lem:inside} we obtain the following set of constraints that must be satisfied:
\begin{align*}
 |\ns Y - \nt Y| & > k \\
 |\ns X -3  - \nt X| & > l  \\
\Sum Y+  k &\geq 8 \text{ and even}\\
\Sum X+  l & \geq 5 \text{ and odd} \\
k +  l+\ns Z+\nt Z&\geq 10 
\end{align*}
The last equation was obtained from the fact that in the kind (20),
the subgraph of $G$ bounded by $K_1$, $K_2$, and $Z$ contains at least two $5^+$-faces.
This system has no solutions. This finishes the proof of Claim~\ref{cl-AA}.
\end{proof}

\begin{claimn}\label{cl-AB}
Let $q_1$ be a configuration of type A and let $q_2$ be a layout where $q_2(f_1)=3$. Then $q_2$ is not a configuration of type B.
\end{claimn}
\begin{proof}
Suppose for a contradiction that $q_1$ gives a configuration of type A and $q_2$ gives a configuration of type B, where  $q_2(f_1) = 3$, hence $f_1=f_2$.
We have four kinds depending on the order of the endpoints of $K_1$ and $K_2$.
The cases are depicted in Figure~\ref{fig-555_cuts_AB}. The kind (22) is not possible if $f_1=f_2$.

\begin{figure}[ht]
\centering
\includegraphics[]{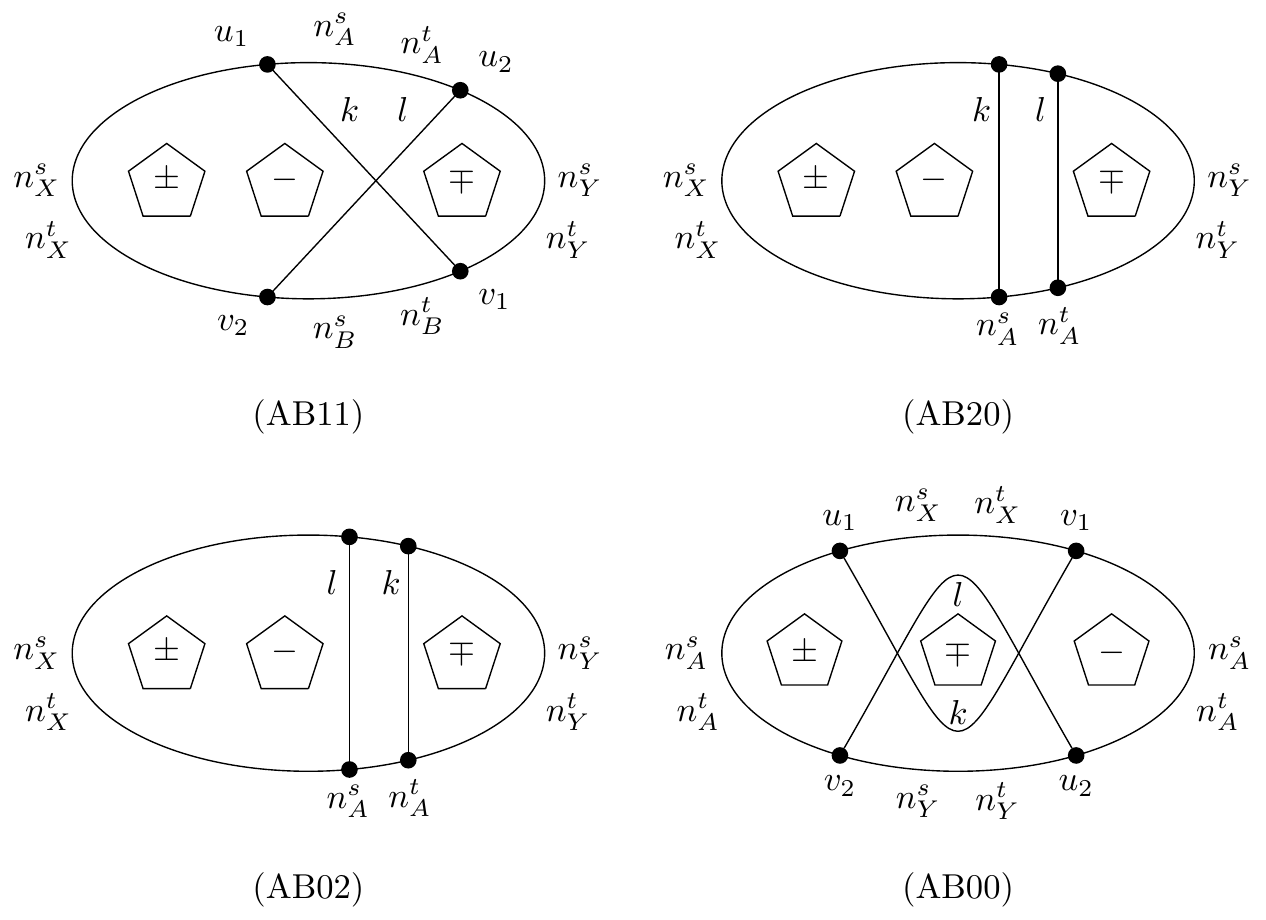}
\caption{Sketches of kinds (11), (20), (02), and (00) for one configuration of type A and one configuration of type B.}
\label{fig-555_cuts_AB}
\end{figure}

Depending on the case, 
from by Lemma~\ref{lemma5} and Lemma~\ref{lem:inside}.
 we obtain a set of constraints that must be satisfied.

\begin{itemize}
\item[(AB11):]
\begin{align*}
   |\ns X + \ns B - \nt X -\nt B| & > k  \\ % \label{555-ab-left-flow1} \\
   |\ns X + \ns A - 6 -\nt X -\nt A| & > l  \\ % \label{555-ab-left-flow2}\\
   \Sum Y+\Sum A +k &\geq 7 \text{ and odd}\nonumber\\
   \Sum X+\Sum Y+k+l & \geq 13 \text{ and odd} \nonumber
\end{align*} 
\end{itemize}

\begin{itemize}
\item[(AB20):]
\begin{align*}
| \ns X- \nt X| &> k \\ % \label{555-ab-middle-flow1} \\
  | \ns Y + 3 - \nt Y| &> l \\ % \label{555-ab-middle-flow2} \\
  \Sum Y+l &\geq 5 \text{ and odd}\nonumber\\
  \Sum X+k & \geq 8 \text{ and even} \nonumber
\end{align*}  
\end{itemize}

\begin{itemize}
\item[(AB02):]
\begin{align*}
   |\ns Y -3 -\nt Y| &> k \\ % \label{555-ab-right-flow1}  \\
   |\ns X - 6 - \nt X| &> l  \\ % \label{555-ab-right-flow2} \\
  \Sum Y+k &\text{ is } \geq 5 \text{ and odd}\nonumber\\
  \Sum  X+l &\text{ is } \geq 8 \text{ and even} \nonumber
\end{align*}  
\end{itemize}

\begin{itemize}
\item[(AB00):]
\begin{align}
   |\ns X -3 -\nt X| &> k \nonumber \\ % \label{555-ab-right-flow1}  \\
   |\ns Y +3 - \nt Y| &> l \nonumber \\ % \label{555-ab-right-flow2} \\
  \Sum X+k & \geq 5 \text{ and odd}\nonumber\\
  \Sum Y+l & \geq 5 \text{ and odd}\nonumber\\
\ns A+\nt A+k+l & \geq 13 \label{555-ab-00-last} 
\end{align}  
\end{itemize}

Inequality~\eqref{555-ab-00-last} comes from the fact that the subgraph bounded by $K_1$, $K_2$ and $A$ must contain all three 5-faces of $G$ in its interior faces.
In addition, we include that $\min\{k_1,k_2,l_1,l_2\}\geq 1$ since $v$ is not a vertex of $C$ and $\min\{k,l\} \geq 2$ since $C$ has no chords.

None of the four sets of constraints has any solution, which is a contradiction.
\end{proof}

By Claim~\ref{cl-AA} and Claim~\ref{cl-AB}, every layout gives a configuration of type B. 
Thus, we know that for each $i \in \{1,2,3\}$, $q_i(f_i) = 3$, and $f_1,f_2,f_3$ are pairwise distinct. 
Let $P_i$ be the subpath of $C$ such that $K_i$ and $P_i$ bound a cycle that contains $f_i$.

Let $i,j \in \{1,2,3\}$ and $i \neq j$.
Based on the order of $u_i,v_i,u_j,v_j$ on $C$, and $K_i$ and $K_j$ we get four
possible kinds (BB00), (BB11), (BB22), (BB20); see Figure~\ref{fig-555-cuts-BB}.
Note that (BB02) is symmetric to what would be (BB20).
\begin{figure}[ht]
\centering
\includegraphics[width=0.8\textwidth]{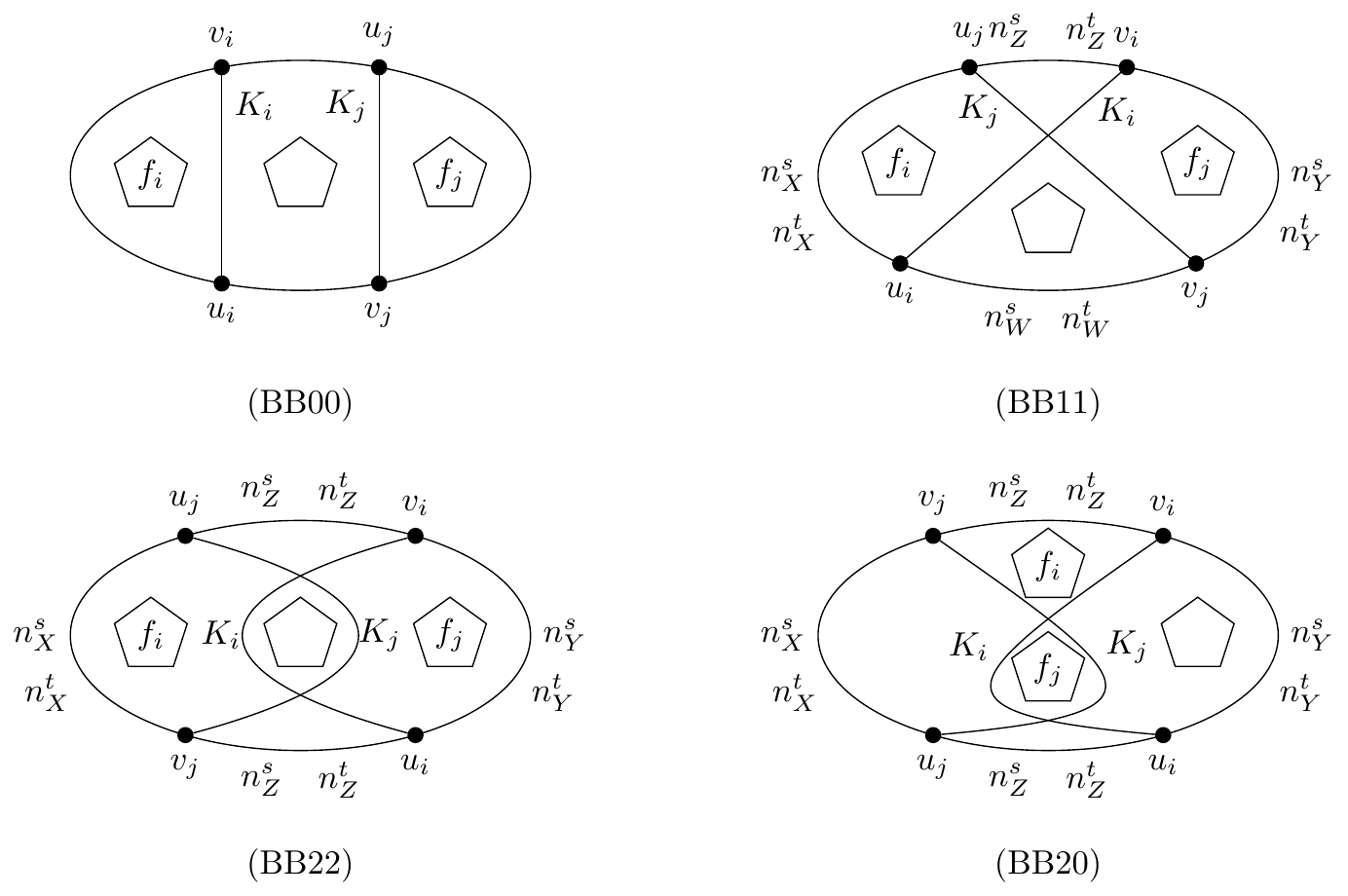}
\caption{Possible configurations of  two cuts of type $B$.}
\label{fig-555-cuts-BB}
\end{figure}

\begin{claimn}\label{cl-555-BB}
For all $i,j\in\{1,2,3\}$ and $i\neq j$ we get that $K_i$ and $K_j$ do not form
(BB22).
\end{claimn}
\begin{proof}
Suppose for a contradiction that $K_i$ and $K_j$ do form (BB22).
See Figure~\ref{fig-555-cuts-BB}~(BB22) for a sketch of the situation.
Let $X$ and $Y$ be $C(v_j, u_j; v_i,u_i)$ and $C(v_i, u_i; v_j,u_j)$, respectively.
Let $Z$ be the edges of $C$ that are in neither $X$ nor $Y$.

Let $t = 6-i-j$.
Since $K_i$ and $K_j$ form (BB22), the subgraph of $G$ bounded by $K_i \cup Y$
contains faces that contain 5-faces $f_j$ and $f_{t}$, 
and the subgraph of $G$ bounded by $K_j \cup X$ contains faces that contain 5-faces $f_i$ and $f_{t}$.
This, Lemma~\ref{lemma5}, and Lemma~\ref{lem:inside} give the following set of constraints.

\begin{align*}
   |\ns Y -6 -\nt Y| &> k_i \\
   |\ns X -6 -\nt X| &> k_j \\
   k_i + \Sum Y  & \geq 8 \text{ and even}\\
   k_j + \Sum X  & \geq 8 \text{ and even}
\end{align*}  
This set of constraints has no solution.
\end{proof}

\begin{claimn}\label{cl-555-BB}
For all $i,j\in\{1,2,3\}$ and $i\neq j$ we get that $K_i$ and $K_j$ do not form
(BB20) or there exist alternative paths that form (BB11) and no new (BB20) is created.
\end{claimn}
\begin{proof}
Suppose  for a contradiction that $K_i$ and $K_j$ form (BB20).
See Figure~\ref{fig-555-cuts-BB}~(BB20) for a sketch of the situation.
Let $X$ and $Y$ be $C(u_j, v_j; v_i,u_i)$ and $C(v_i, u_i; u_j,v_j)$, respectively.
Let $Z$ be the edges of $C$ that are in neither $X$ nor $Y$.

First we will obtain a few potential solutions. The first four inequalities follow from Lemmas~\ref{lemma5} and \ref{lem:inside}. The inequality \eqref{eq14} comes from the description of (BB20) where 
the subgraph of $G$ bounded by $K_i,K_j$, and $Z$ contains at least three interior faces where at least two are 5-faces. The inequality \eqref{eq7} comes from (BB20) saying that $X$ and $K_j$ do not form the boundary of $f_j$.
\begin{align}
   |\ns Y + 3 - \nt Y| &> k_i \nonumber\\
   |\ns X + 3 -\nt X| &> k_j \nonumber\\
   k_j + \Sum X  &\geq 5 \text{ and odd}\nonumber \\
   k_i + \Sum Y  &\geq 8 \text{ and even}\nonumber\\
   k_i+k_j+\Sum Z &\geq 14 \label{eq14}\\
   k_j+\Sum X &\geq 7 \label{eq7} 
\end{align}  
This set of constraints has the following four solutions.
{
\begin{center}
\setlength{\tabcolsep}{8pt}
\renewcommand{\arraystretch}{1.2}
\begin{tabular}{|c|c|c|c|c|c|c|c|}
\hline
$\ns X$ & $\nt X$ & $\ns Y$ & $\nt Y$ & $\ns Z$ & $\nt Z$ & $k_i$ & $k_j$ \\
\hline
3&0&0&3&3&0&7&4  \\ \hline
4&0&0&3&2&0&7&5  \\ \hline
5&0&0&3&1&0&7&6  \\ \hline
6&0&0&3&0&0&7&7  \\ \hline
\end{tabular}
\end{center}
}
Notice that in all the solutions $k_i+k_j+\Sum Z = 14$. 
Hence the subgraph of $G$ bounded by $K_i,K_j$, and $Z$ has two 5-faces and one 4-face.
We create a more detailed instance where we split $Z$ into two paths $C(v_j,v_i;u_j,u_j)$ that we
keep calling $Z$ and $C(u_i,u_j;v_j,v_i)$ that we call $W$. Moreover, we partition $K_i$ and $K_j$
into three subpaths of lengths $i_1,i_2,i_3$ and $j_1,j_2,j_3$ respectively.
See Figure~\ref{fig-555-cuts-BB20}.
 This leads to the following constraints, where the first six are the same as before.
 \begin{align}
   |\ns Y + 3 - \nt Y| &> k_i \nonumber\\
   |\ns X + 3 -\nt X| &> k_j \nonumber\\
   k_j + \Sum X  & \geq 5 \text{ and odd}\nonumber \\
   k_i + \Sum Y  & \geq 8 \text{ and even}\nonumber\\
   k_i+k_j+\Sum Z &\geq 14 \nonumber\\
   k_j+\Sum X &\geq 7  \nonumber\\
   i_1+j_2+i_3+ \Sum Y   & \geq 5 \text{ and odd}\\
   i_2+j_2  &\geq 5 \text{ and odd}\\
   i_1+j_1 + \Sum Z  & \geq 5 \text{ and odd}
\end{align}  
The system has 14 solutions. 
Create a path $K_j'$ from $K_j$ by dropping the piece corresponding to $j_3$
and replacing it by $i_3$ and potentially deleting repeated edges. The path $K_j'$ is a path with endpoints in $C$ and it makes $f_j$ lonely.
Moreover, all 14 solutions satisfy $|\ns X + \ns W + 3 -\nt X -\nt W| > j_1+j_2+i_3$. 
Hence $K_j'$ can be used instead of $K_j$, and $K_j'$ and $K_i$ form configuration (BB11).
\begin{figure}[ht]
\centering
\includegraphics{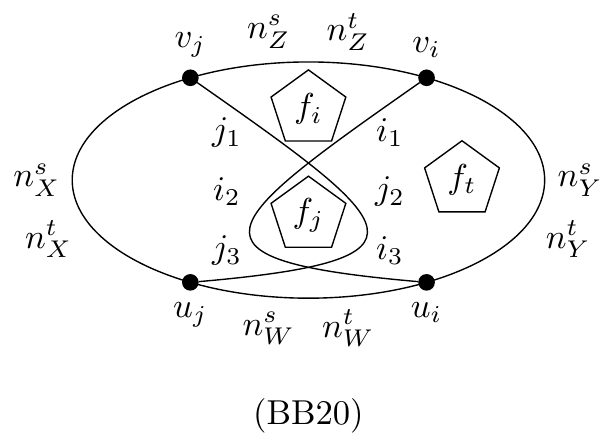}
\caption{More details for configuration (BB20).}
\label{fig-555-cuts-BB20}
\end{figure}

Finally, we need to show that no new (BB20) or (BB02) is created by replacing $K_j$ by $K_j'$.
All 14 solutions satisfy $i_3+j_3+\ns W+\nt W = 4$, $i_2+j_2=5$, $i_1+j_1+\ns Z +\nt Z = 5$, $\Sum Y = 3$, and $\Sum Z \leq 2$.
Hence the subgraph of $G$ induced by $W$, $Z$, $K_i$, and $K_j$ contains a 4-face and two 5-faces as internal faces and one of the 5-faces is sharing with $C$ vertices $v_j$ and $v_i$. 

Let $K_a$ and $K_b$ form (BB20) or (BB20) for some $a,b \in \{1,2,3\}$.
The previous paragraph implies that for any $c \in \{a,b\}$ 
one of the edges of $K_c$  incident to $v_c$ and $u_c$ is incident to a 4-face and the
other edge is incident to a 5-face $f_a$ or $f_b$. Moreover, one of $f_a$ and $f_b$ is disjoint from $C$
and the other one is sharing at most two edges with $C$.

Let $t = 6-i-j$ and $K_t$ be in $\{K_1,K_2,K_3\}$ with endpoints $u_t$ and $v_t$.
Suppose for contradiction that a new (BB20) or (BB02) is created by replacing $K_j$ by $K_j'$.
Since $K_i$ is not changed, the new (BB20) or (BB02) is formed by $K_j'$ and $K_t$.
Hence $K_j$ and $K_t$ is neither (BB20) nor (BB02).
The new  (BB20) or (BB02) must satisfy the constraints from the previous paragraph.
Since the edge of $K_j'$ incident to $v_j$ is incident to a 4-face and $f_i$,
the edge $e$ of $K_j'$ incident to $u_i$ must be incident to $f_t$.
Notice that $e$ is also incident to a 4-face $h$ that is incident to $W$. 
Hence $f_t$ must be on the opposite side of $e$ than $h$.
Let  $x \in \{u_t,v_t\}$ be incident to an edge of $K_t$ that is incident to $f_t$.
Since $\Sum Y = 3$ and $f_t$ is sharing at most two edges with $C$, we obtain that
$x \in Y$ and the order around $C$ is $u_iv_jx$. 
Since $K_j'$ and $K_t$ form (BB20) or (BB02) and we know the order for $x$, the order of the endpoints of $K_j'$ and $K_t$ is  $u_iv_jv_tu_t$. Hence $K_j'$ and $K_t$ form (BB02).
Observe that the order of endpoints of $K_j$ and $K_t$ is $u_jv_jv_tu_t$. Hence $K_j$  and $K_t$ form (BB02), a contradiction.
\end{proof}

\begin{claimn}\label{cl-555-BB11}
For all $i,j\in\{1,2,3\}$ and $i\neq j$ if $K_i$ and $K_j$ form (BB11) then they have a common point.
\end{claimn}
\begin{proof}
In order to apply Lemma~\ref{lem:common} we need to verify that 
$\max\{|K_i|,|K_j|\} \leq 7$ and if  $|K_i| = |K_j|=7$ then $K_i$ and $K_j$ have common endpoints.
Let $K_i$ and $K_j$ form (BB11), see Figure~\ref{fig-555-cuts-BB} (BB11) for illustration.
Let $X$, $Z$, $Y$, and $W$ be $C(u_i,u_j;v_i,v_j)$, $C(u_j,v_i;v_j,u_i)$, $C(v_i,v_j;u_i,u_j)$, and $C(v_j,u_i;u_j,v_i)$, respectively. 
Denote $|K_i|$ and $|K_j|$ by $k_i$ and $k_j$, respectively.
Lemma~\ref{lemma5} and Lemma~\ref{lem:inside} imply that the following constraints are satisfied.
 \begin{align}
   |\ns X + \ns Z + 3 - \nt X - \nt Z| &> k_i \nonumber\\
   |\ns Y + \ns Z + 3 - \nt Y - \nt Z| &> k_j \nonumber\\
   \Sum X + \Sum Z + k_i  & \geq 7 \text{ and odd}\nonumber \\
   \Sum Y + \Sum Z + k_j  & \geq 7 \text{ and odd}\nonumber
\end{align} 
All solutions to these constraints satisfy that $\max\{k_i,k_j\} \leq 7$. Moreover,
if $k_i=k_j=7$ then $\Sum X + \Sum Y = 0$. Hence  Lemma~\ref{lem:common} applies and
there is a common point.
\end{proof}

Now we know that we have only configurations (BB00) and (BB11) with common points.

Denote the length of the path $K_1$, $K_2$, and $K_3$ by $k$, $l$, and $m$, respectively.
We will use $k_1$, $k_2$, $k_3$ to denote the lengths of subpaths of $k$ if some of the paths form (BB11); $l_1$, $l_2$, $l_3$, $m_1$, $m_2$, $m_3$ will be used similarly. 
See Figure~\ref{fig-555-cuts-BBB}.

If there is a pair of layouts giving configuration (BB00), we distinguish the following cases:\begin{itemize}
\item[(B1)] all pairs form (BB00).
\item[(B2)] one pair  forms (BB11).
\item[(B3)] two pairs form (BB11)
\end{itemize}
If all three pairs of layouts give (BB11), then we define $v_K$ and $v_L$ to be the common point of $K_3$ with $K_1$ and $K_2$, respectively.
The vertex $v_K$ is \emph{before} $v_L$ if $v_K$ appears before $v_L$ when
traversing the cycle formed by $K_3$ and $P_3$ in the clockwise order and the starting
point is on $C$.
\begin{itemize} 
   \item[(B4)] $P_1$, $P_2$, and $P_3$ have a common edge and $v_K$ is before $v_L$ or $v_K = v_L$.
   
   \item[(B5)] There is no common edge of $P_1$, $P_2$, and $P_3$ and  $v_K$ is before $v_L$ or $v_K=v_L$.
   \item[(B6)] There is no common edge of $P_1$, $ P_2$, and $P_3$, and  $v_L$ is before $v_K$ and $v_K \neq v_L$
   \item[(B7)] $P_1$, $P_2$, and $P_3$ have a common edge and $v_L$ is before $v_K$ and $v_K \neq v_L$.
   \end{itemize}
See Figure~\ref{fig-555-cuts-BBB} for an illustration of the cases (B1)--(B7).
Since one layout may contain several different configurations of type B, pick  $K_1$, $K_2$, $K_3$ such that
the number of (B11) pairs is minimized.

\begin{figure}[ht]
\centering
\includegraphics[width=0.7\textwidth]{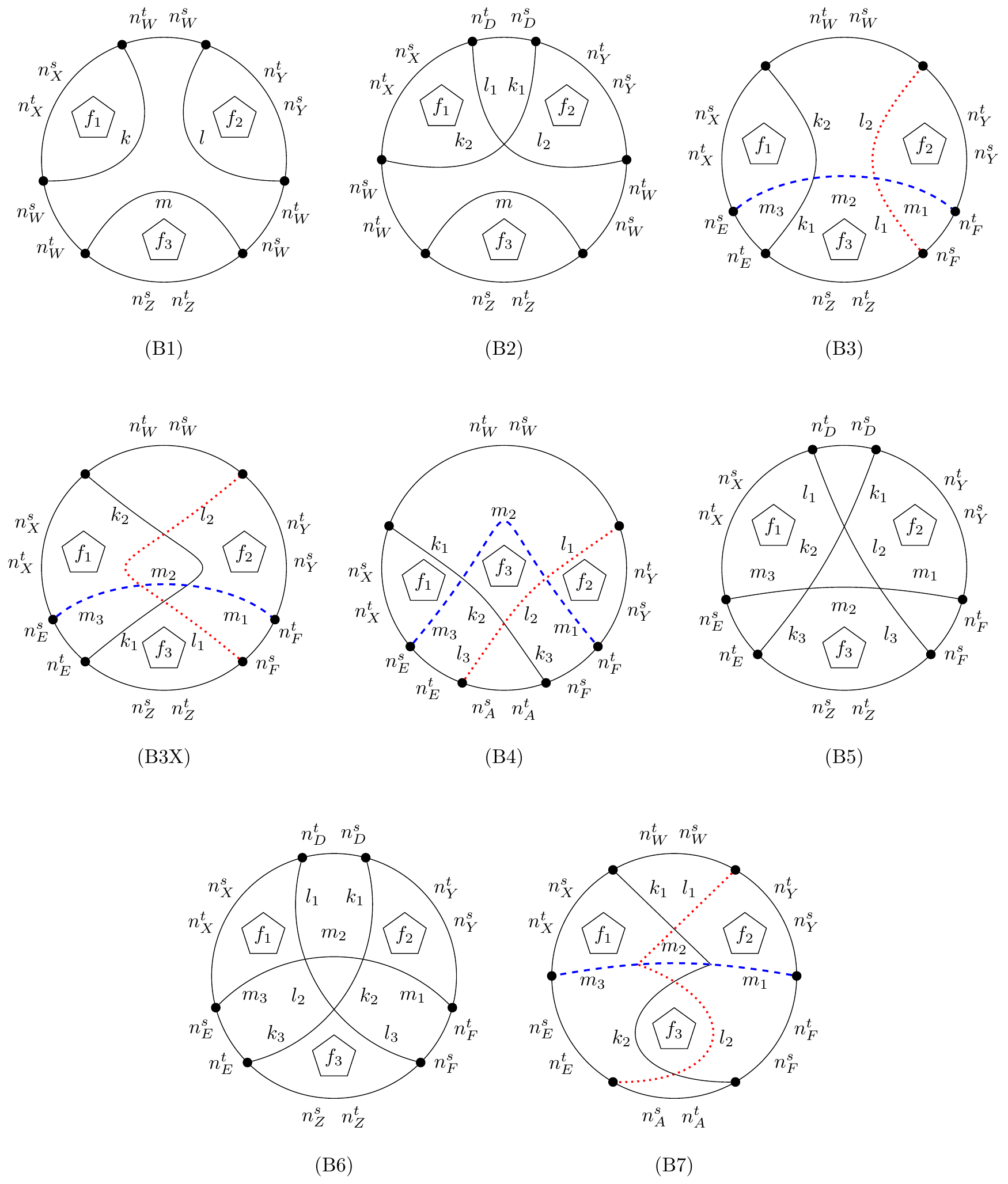}
\caption{Possible configurations of cuts $K_1$, $K_2$, $K_3$.}
\label{fig-555-cuts-BBB}
\end{figure}

Next we give constraints for each of the cases (B1)--(B7). Solutions to these constraints were obtained by simple computer programs.
Critical graphs obtained from (B$i$) are depicted in Figure~\ref{fig-all-critical-graphs} as (B$ij$) for all $i,j$.

Endpoints of $K_1$, $K_2$, and $K_3$ partition $C$ into several internally disjoint paths.
The paths have names in $\{X,Y,Z,W,A,D,E,F\}$ with exception of $W$ in (B1) and (B2),
where $W$ refers to a union of up to three and two paths, respectively.
To simplify the write-up we refer the reader to Figure~\ref{fig-555-cuts-BBB} for the labelings of the paths.

\begin{claimn}\label{cl-B1}
 The configuration (B1) results in a critical graph where every 5-face shares at least two edges with the boundary. 
Moreover, in every non-extendable $3$-coloring of the outer face, every 5-face contains two source edges.
\end{claimn}
\begin{proof}
 We refer the reader to Figure~\ref{fig-555-cuts-BBB} (B1) for the labelings of the paths.
 By Lemma~\ref{lemma5} we get the first three equations and by Lemma~\ref{lem:inside} we get the remaining equations.
\begin{align*}
   |\ns X + 3 -\nt X| &> k , \\
   |\ns Y + 3 -\nt Y| &> l , \\
   |\ns Z + 3 -\nt Z| &> m , \\
   k+\Sum X & \geq 5 \text{ and odd,}\\
   l+\Sum Y & \geq 5 \text{ and odd,}\\
   m+\Sum Z & \geq 5 \text{ and odd,}\\
\end{align*}  
In addition, we also include some constraints to break symmetry; for example $\nt X + \ns X \geq  \nt Y + \ns Y \geq \nt Z + \ns Z$.
All solutions to this system of equations are in Table~\ref{tab:solBBB1}. 
\begin{table}[ht]
\begin{center}
\setlength{\tabcolsep}{8pt}
\renewcommand{\arraystretch}{1.2}
\begin{tabular}{|c|c|c|c|c|c|c|c|c|c|c|}
\hline
$\ns X$ & $\nt X$ & $\ns Y$ & $\nt Y$ & $\ns Z$ & $\nt Z$ & $\ns w$ & $\nt W$ & $k$ & $l$ &  $m$\\
\hline
2&0&2&0&2&0&0&3&3&3&3\\
\hline
2&1&2&0&2&0&0&2&2&3&3\\
\hline
2&1&2&1&2&0&0&1&2&2&3\\
\hline
2&1&2&1&2&1&0&0&2&2&2\\
\hline
\end{tabular}
\end{center}
\caption{Solutions from Claim~\ref{cl-B1}.}\label{tab:solBBB1}
\end{table}
 By inspecting the solutions from Table~\ref{tab:solBBB1}, we conclude that they satisfy the statement of the claim.
\end{proof}

For the remaining cases, we give the sets of constraints but we skip detailed justification since they all come from the description of the configurations, Lemma~\ref{lemma5}, Lemma~\ref{lem:inside}, and the fact that $C$ has no chords.
We provide computer programs online for solving the sets of equations and to help with checking the solutions.

The description of the configuration is the following, see Figure~\ref{fig-555-cuts-BBB}.
If there is exactly one (B11) pair, then we get configuration (B2), where we assume it is pair $q_1$ and $q_2$.

If there are two (B11) pairs, then assume that $q_3$ is in both pairs. 
There are two common points on $K_3$, where one is shared with $K_1$ and the other is shared with $K_2$.
Depending on the order of these points we get either (B3) or (B3X).
The last option is that all three pairs are (B11).
By considering the order of the endpoints of $K_1,K_2,K_3$ and the order of the common points on $K_3$, we get (B4)--(B7).

\begin{claimn}\label{cl-B2}
 Configurations (B2)--(B7) result in critical graphs (B21)--(B52).
 Every graph in Figure~\ref{fig-all-critical-graphs} represents several graphs
 that can be obtained from the depicted graph by identifying edges and vertices
 and by filling every face of even size by a quadrangulation with no separating
 4-cycles. Moreover, the 5-faces in (B21) and (B22) that share two edges with $C$
 can be moved along $C$ as long as they stay neighboring with a region with three 
 sink edges.
\end{claimn}
\begin{proof}[Proof Outline:]
We slightly abuse notation and use $k_i,l_i,m_i$ for subpaths of $K_1,K_2,K_3$ respectively as well
as for lengths of these subpaths, where $i \in \{1,2,3\}$.
For a path in $\{X,Y,Z,W,A,D,E,F\}$, we use its lower case letter to denote its length.

For each case we include constraints that all three layouts give configurations of type $B$
using Lemma~\ref{lem:inside} and Lemma~\ref{lemma5} analogously to Claim~\ref{cl-B1}.
In addition, we add the following set of constraints depending on the case:
\begin{itemize}
\item[(B2):]
\begin{align*}
x + k_2 + l_1 &\geq 5 \text{ and odd} &
y + k_1 + l_2 &\geq 5 \text{ and odd}  \\
z+w+k_2+l_2 &\geq 7 \text{ and odd} \text{ if }w > 0 &
x+d+y+k_2+l_2 &\geq 8 \text{ and even} \\
x+y+z+w+l_1+k_1 &\geq 9 \text{ and odd} 
\end{align*}
\end{itemize}

\begin{itemize}
\item[(B3):]
\begin{align*}
e+x+k_1+k_2 &\geq 7  &
f+y+l_1+l_2 &\geq 7 \\
y+m_1+l_2 &\geq 5 \text{ and odd} &
z+k_1+m_2+l_1 &\geq 5 \text{ and odd} \\
x+m_3+k_2 &\geq 5 \text{ and odd} &
y+z+f+k_1+m_2+l_2 &\geq 8 \text{ and even} 
\end{align*}
\end{itemize}

\begin{itemize}
\item[(B3X):]
\begin{align*}
\min\{m_1,m_2,m_3,k_1,k_2,l_1,l_2\} &\geq 1 &  k_1+l_1+z  &\geq 6 \\
\text{if } l_2 = 1\text{ then } x+m_3+k_2 &\geq 6  &   y+m_1+m_2+l_2 &\geq 5 \text{ and odd} \\
\text{if } k_2 = 1\text{ then } y+m_1+l_2 &\geq 6 &  x+k_2+m_2+m_3 &\geq 5 \text{ and odd} 
\end{align*}
\end{itemize}

\begin{itemize}
\item[(B4):]
\begin{align*}
x+e+l_3+k_1+k_2&\geq 5  \text{ and odd} &
  y+f+k_3+k_2+m_2+l_1 &\geq 8  \text{ and even}  \\
k_2+l_2+m_2&\geq 5  \text{ and odd} &
f+y+w+x+m_3+k_2+k_3 &\geq 9  \text{ and odd}  \\
x+k_1+m_3 &\geq 5  \text{ and odd} &
e+x+k_1+m_2+l_2+l_3 &\geq 8  \text{ and even} \\
y+l_1+m_1 &\geq 5  \text{ and odd} &
y+w+x+e+l_3+l_2+m_1 &\geq 9  \text{ and odd} \\
& & f+y+w+x+e+l_3+k_3 &\geq 9  \text{ and odd}
\end{align*}
\end{itemize}

\begin{itemize}
\item[(B5):]
\begin{align*}
 \min\{k_2,l_2,m_2\} &\geq 1 \text{ or }k_2=l_2=m_2 = 0
\end{align*}
\begin{align*}
 y+k_1+l_2+m_1 &\geq 5  \text{ and odd} &
 x+l_1+k_2+m_3 &\geq 5  \text{ and odd} \\
 z+k_3+m_2+l_3 &\geq 5  \text{ and odd} &
 y+f+l_3+l_2+k_1  &\geq 7  \text{ and odd} \\
 x+e+k_3+k_2+l_1  &\geq 7  \text{ and odd} &
 m_1+m_2+k_3+f+z &\geq 7  \text{ and odd} \\
 m_3+m_2+l_3+e+z &\geq 7  \text{ and odd} &
 l_1+l_2+m_1+d+y  &\geq 7  \text{ and odd}
\end{align*}
\end{itemize}

\begin{itemize}
\item[(B6):]
\begin{align*}
m_2 &\geq 1 &
 y+k_1+m_1 &\geq 5  \text{ and odd} \\
 x+l_1+m_3 &\geq 5  \text{ and odd} &
 z+k_3+l_3 &\geq 5  \text{ and odd}
\end{align*}
\end{itemize}

\begin{itemize}
\item[(B7):]
\begin{align}
k_1+l_1+m_1+m_3+x+y &\geq 10  \label{eq7:a} \\
e+x+w+y+f+l_2+k_2-5 &\geq 9 \label{eq7:b}
\end{align}
\end{itemize}

We enumerated all solutions to all seven sets of constraints, and we checked that the resulting graphs are depicted in Figure~\ref{fig-all-critical-graphs}.
In order to eliminate mistakes in computer programs, we have two implementations
by different authors and we checked that they give identical results.
Sources for programs for cases (B2)--(B7) together with their outputs
can be found on arXiv and at \ourURL.

\begin{figure}
\begin{center}
\includegraphics[width=\tempL\textwidth,page=1]{fig-criticals/fig-555-critical}
\includegraphics[width=\tempL\textwidth,page=2]{fig-criticals/fig-555-critical}
\includegraphics[width=\tempL\textwidth,page=3]{fig-criticals/fig-555-critical} 
\includegraphics[width=\tempL\textwidth,page=7]{fig-criticals/fig-555-critical}
\includegraphics[width=\tempL\textwidth,page=8]{fig-criticals/fig-555-critical}

\includegraphics[width=\tempL\textwidth,page=9]{fig-criticals/fig-555-critical}
\includegraphics[width=\tempL\textwidth,page=4]{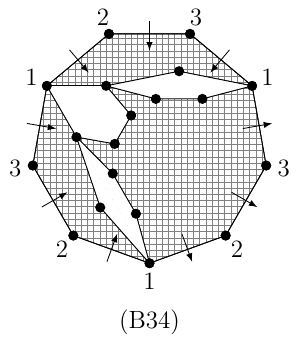} 
\includegraphics[width=\tempL\textwidth,page=5]{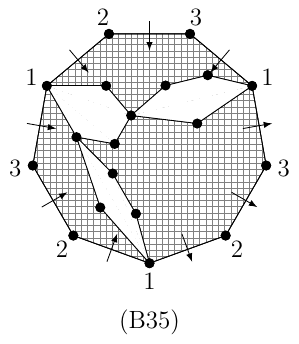}
\includegraphics[width=\tempL\textwidth,page=6]{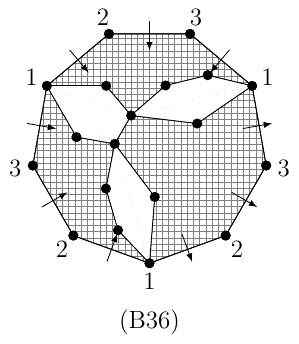}
\includegraphics[width=\tempL\textwidth,page=10]{fig-criticals/fig-555-critical}

\includegraphics[width=\tempL\textwidth,page=11]{fig-criticals/fig-555-critical}
\includegraphics[width=\tempL\textwidth,page=12]{fig-criticals/fig-555-critical}
\includegraphics[width=\tempL\textwidth,page=13]{fig-criticals/fig-555-critical}
\includegraphics[width=\tempL\textwidth,page=15]{fig-criticals/fig-555-critical}
\includegraphics[width=\tempL\textwidth,page=14]{fig-criticals/fig-555-critical}
\end{center}
\caption{All solutions to cases (B1)--(B7).
}
\label{fig-all-555-solutions}
\end{figure}

The most general solution for each of the sets of equations is depicted in Figure~\ref{fig-all-555-solutions}.  
Notice that (B3X), (B6), and (B7)  have no solutions.
In (B7), inequality \eqref{eq7:a} comes from a subgraph having two faces where each contains a 5-face in the interior and \eqref{eq7:b} comes from Lemma~\ref{lem:inside} and the $-5$ appears due to $k_2$ and $l_2$ enclosing $f_3$. 
Observe that (B34), (B35), and (B36) are
special cases of  (B41), (B42), and (B43), respectively. Hence we dropped (B34), (B35), and (B36) from Figure~\ref{fig-all-critical-graphs}. 
One can think of (B41), (B42), and (B43) as being obtained from (B34), (B35), and (B36) by duplicating a subpath $P$ of $C$  where all dual edges of $P$ are oriented inside. 
Notice that by using this operation, (B44) could be obtained from (B21), also (B22) from (B11) and (B41) from (B22). 
We suspect that it is part of a more general description of $C$-critical graphs, where $C$ is larger.

We think the case (B4) is the most complicated case.
We again used the trick to identify general solutions quickly by observing that regions bounding faces contain only the face
and obtained seven solutions. We include sketches of the solutions generated by our program in Figure~\ref{fig-B4}. Although there are
seven solutions, they give only four distinct cases due to some vertex identifications.

\def\tempL{0.22}
\begin{figure}
\begin{center}
\includegraphics[width=\tempL\textwidth,page=1]{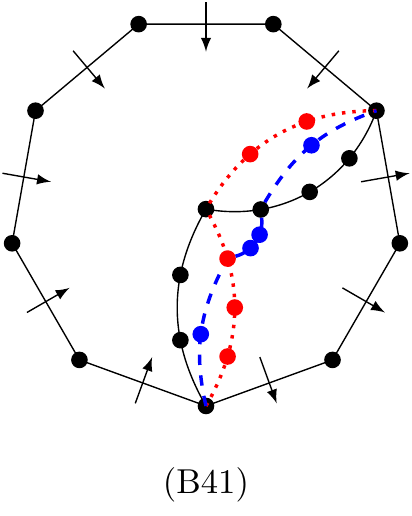}
\includegraphics[width=\tempL\textwidth,page=2]{fig-B4}
\includegraphics[width=\tempL\textwidth,page=3]{fig-B4}
\includegraphics[width=\tempL\textwidth,page=4]{fig-B4}
\includegraphics[width=\tempL\textwidth,page=5]{fig-B4}
\includegraphics[width=\tempL\textwidth,page=6]{fig-B4}
\includegraphics[width=\tempL\textwidth,page=7]{fig-B4}
\end{center}
\caption{Sketches of solutions to case (B4) generated by our program.  The style of paths $K_1$, $K_2$, and $K_3$ correspond to the style in Figure~\ref{fig-555-cuts-BBB}~(B4). 
}
\label{fig-B4}
\end{figure}

\end{proof}

This finishes the proof of Lemma~\ref{lall5}.
\end{proof}

%%%%%%%%%%%%%%%%%%%%%%%%%%%%%%%%%%%%%%%%%%%%%%%%%%%%%%%%%
%														%
%														%
%														%
%														%
%														%
%														%
%						End of proof					%
%														%
%														%
%														%
%														%
%														%
%														%
%														%
%%%%%%%%%%%%%%%%%%%%%%%%%%%%%%%%%%%%%%%%%%%%%%%%%%%%%%%%%

\section{Acknowledgements}

We would like to thank Zden\v{e}k Dvo\v{r}\'{a}k for fruitful discussions and  
we are very grateful to anonymous referee who spotted numerous mistakes
in the paper and suggested simplification to the proofs.

 This work was supported by the European Regional Development Fund (ERDF), project NTIS - New Technologies for the Information Society,
 European Centre of Excellence, CZ.1.05/1.1.00/02.0090.

 The first author was supported by the National Research Foundation of Korea (NRF) grant funded by the Korea government (MSIP) (NRF-2015R1C1A1A02036398).

 The second and third authors were supported by project P202/12/G061 of the Grant Agency of the Czech Republic.

 The last author was supported by NSF grants DMS-1266016 and DMS-1600390.
 
 A preliminary version of this paper without a complete proof was published in proceedings on IWOCA 2014~\cite{cyclelncs}.
 
\bibliographystyle{abbrv}
\bibliography{9cyc}

\end{document}